\documentclass[a4paper,10pt]{article}
\usepackage{amsthm,amsfonts}
\usepackage{mathrsfs}
\usepackage{stmaryrd}
\usepackage{amsmath,amssymb,amsfonts}
\usepackage{array,booktabs,longtable}
\usepackage{sectsty}
\usepackage{hyperref}
\sectionfont{\large}
\subsectionfont{\normalsize}
\usepackage{hyperref}
\usepackage[all]{xy}
\usepackage{graphicx}
\usepackage[dvips]{color}

\newtheorem{theorem}{Theorem}[section]
\newtheorem{definition}{Definition}[section]
\newtheorem{proposition}{Proposition}[section]
\newtheorem{remark}{Remark}[section]
\newtheorem{lemma}{Lemma}[section]

\newtheorem{corollary}{Corollary}[section]

\usepackage{hyperref}

\newcommand{\rd}{{\text{\rm d}}}

\newcommand{\bv}{\mathbf v}
\newcommand{\bu}{\mathbf u}
\newcommand{\bF}{\mathbb F}
\newcommand{\f}{\mathbf f}

\newcommand{\z}{\mathbf z}
\newcommand{\x}{\mathbf x}
\newcommand{\bw}{\mathbf w}

\newcommand{\dist}{\mbox{dist}}
\newcommand{\mM}{\mathcal M}
\newcommand{\mB}{\mathcal B}

\newcommand{\mC}{\mathcal C}

\newcommand{\mF}{\mathcal F}
\newcommand{\mP}{\mathcal P}
\newcommand{\mY}{\mathcal Y}
\newcommand{\mU}{\mathcal U}

\newcommand{\wsconv}{\stackrel{*}{\rightharpoonup}}
\newcommand{\wconv}{\stackrel{w}{\rightarrow}}

\begin{document}

\title{On the convergence of statistical solutions of the 3D
Navier-Stokes-$\alpha$ model as $\alpha$ vanishes}
\author{Anne C. Bronzi 
\thanks {annebronzi@gmail.com.}
\and Ricardo M. S. Rosa \thanks {rrosa@im.ufrj.br}}


\date{}
\maketitle

\begin{abstract}
In this paper statistical solutions of the 3D Navier-Stokes-$\alpha$ model with
periodic boundary condition are considered. It is proved that under certain
natural conditions statistical solutions of the 3D Navier-Stokes-$\alpha$ model
converge to statistical solutions of the exact 3D Navier-Stokes equations as
$\alpha$ goes to zero. The statistical solutions considered here arise as
families of time-projections of measures on suitable trajectory spaces.
\end{abstract}

\smallskip


\section{Introduction}

In this work we investigate the convergence of statistical solutions of the
three-dimensional Navier-Stokes-$\alpha$ model to statistical solutions of the
three-dimensional
Navier-Stokes equations, as $\alpha$ goes to zero. We consider the equations
with periodic boundary conditions and zero space average, and the statistical
solutions are considered in the sense recently introduced by Foias, Rosa and
Temam in
\cite{FRT,frtcr2010a}. 

Most of the knowledge concerning turbulent flows are rooted in heuristic and
phenomenological arguments. One of the fundamental observations is that,
although quite irregular, turbulent flows display a certain order in a
statistical sense, so that mean quantities and other low order moments are
usually more regular and predictable. The statistical solutions that we
investigate are mathematical objects used to address the statistical properties
of the flow in a rigorous mathematical way directly from the equations of
motion. 

The Navier-Stokes equations have been widely used as a model for Newtonian
turbulent flows  (see e.g. \cite{Bat53, Hin75, MonYag75, Ten72}) and a rigorous
mathematical formalization of the statistical theory of turbulence came with the
introduction of statistical solutions for these equations. Several statistical
estimates can be obtained rigorously using statistical solutions; see for
instance \cite{foias74, cfm1994, bcfm1995, foias97, fursikov1999, FMRT01b,
RRT}. Therefore, a better understanding of statistical solutions  is crucial for
a rigorous mathematical theory of turbulence.

The concept of statistical solutions was first introduced by Foias and Prodi
\cite{foias72,fp1976} in the early 1970's (see also an earlier
related mathematical work by Hopf \cite{Hopf}). They considered as a statistical
solution a family of measures on the phase space satisfying a Liouville-type
equation together with some regularity conditions. Some years later
Vishik and Fursikov  \cite{vishikfursikov77} introduced a
different type of statistical solutions, given by measures on suitable
trajectory spaces (see also \cite{vishikfursikov78,vishikfursikov88}).  More
recently, in \cite{FRT,frtcr2010a}, Foias, Rosa and Temam elaborated a notion
of statistical solutions which, in some sense, relates the two previous notions
and has better analytical properties than the previous ones. More precisely,
they considered a Borel probability measure on the space $\mathcal C([0,T],
H_w)$ which is carried by the trajectory space of Leray-Hopf weak solutions. The
family of time projections of such a measure gives rise to a statistical
solution in the sense of Foias-Prodi that possess additional analytical
properties. 

We also mention a work due to Capinski and Cutland, \cite{CM92}, preceding the
works of Foias, Rosa and Temam, in which the authors prove, using non-standard
analysis, the existence of space-time statistical solutions (and of individual
weak solutions) defined in the same spirit as that in \cite{FRT,frtcr2010a},
with the main difference that they use a slightly different definition of weak
solutions upon which the statistical solutions are built.

There is also a large literature on the stochastic version of the Navier-Stokes
equations, which relies on some ideas from the Vishik-Fursikov formulation; see e.g.
\cite{flandoli}.

A few other equations have also been considered as suitable models for
turbulent fluid motions in specific aspects. For instance, the
Navier-Stokes-$\alpha$ model is known as a good approximation model for
well-developed turbulent flows (see \cite{cfhotw98,cfhotw99a,cfhotw99b} for the
cases of infinite pipes and channels). Also, the 3D Navier-Stokes-Voigt
equations are known as an appropriated model for direct numerical simulations
of turbulent flows in statistical equilibrium. For a survey on approximating
models of the Navier-Stokes equations and their properties we refer the reader
to the Introduction given in \cite{HLT}. Regarding the importance of statistical
solutions to yield statistical estimates rigorously, we highlight the work of
Ramos and Titi, \cite{RT}. The authors derived statistical properties of the
invariant measures associated with the solutions of the 3D Navier-Stokes-Voigt
equations and established the convergence of probability invariant measures
associated with the 3D Navier-Stokes-Voigt equations to stationary statistical
solution of the 3D Navier-Stokes equations. With this result, they argue,
via statistical estimates obtained rigorously, that the 3D Navier-Stokes-Voigt
is in fact a reliable subgrid scale model for direct numerical simulations of
turbulent flows.

We focus our study on the Navier-Stokes-$\alpha$ model and on the statistical
solution defined by Foias, Rosa and Temam, namely a
family of time projections of a measure carried by the trajectory space of the
Leray-Hopf weak solutions; see Section \ref{sec1.6}. The Navier-Stokes-$\alpha$
model (also known as Camassa-Holm equations) was introduced by  Chen \textit{et
al}, in \cite{cfhotw98}. This model is a regularized approximation of the 3D
Navier-Stokes equations such that, in some terms of the equation, the velocity
field is replaced by a smoother (filtered) velocity field depending on a small
parameter $\alpha>0$; see Section \ref{sec1.4}. As observed in
\cite{foiasholmtiti2001}, this regularized approximation introduces an energy
penalty that inhibits the creation of smaller and smaller excitations below the
length scale $\alpha$. More precisely, it was proved in \cite{foiasholmtiti2001}
that the wavenumber spectrum of the translational kinetic energy for the
Navier-Stokes-$\alpha$ model rolls off as $\kappa ^{-3}$ for $\kappa \alpha>1$
instead of continuing along the Kolmogorov scaling law, $\kappa^{-5/3}$, which
is followed for $\kappa \alpha<1$.

In this article we prove that, under certain natural conditions,  statistical solutions
of the Navier-Stokes-$\alpha$ model converge to statistical solutions of the
Navier-Stokes equations as $\alpha$ goes to zero. We consider both
stationary and time-dependent statistical solutions. We point out the importance
of stationary statistical solutions in the study of turbulence in statistical
equilibrium in time, while, time-dependent statistical
solutions are useful in the study of evolving or decaying turbulence. 

Another important fact is
that the notion of stationary statistical solutions provides a generalization of
the
notion of invariant measures of semigroups. In the case of the 3D-Navier-Stokes
equations this is currently needed since the well-posedness of the
equations has not been established. In the two-dimensional case, in which the
Navier-Stokes equations have a well-defined semigroup, the two notions have been
proved to agree with each other. 

The definition of statistical solution for
the Navier-Stokes-$\alpha$ model is inspired by the corresponding definition for
the
Navier-Stokes equations, being the family of projections in time of a Borel
probability measure in a suitable trajectory space and carried by the set of
individual solutions of the equation. The natural trajectory space for the
Navier-Stokes-$\alpha$ model is $\mC([0,T],H)$, which is included in
$\mC([0,T],H_w)$, which is the natural space for the Navier-Stokes equations, so
that both statistical solutions can be regarded as projections of measures on
the same space. In this way, both statistical solutions can be viewed as
mathematical objects of the same type, allowing a direct comparison between
them, and yielding a natural framework for studying the convergence of the
corresponding statistical solutions as $\alpha$ goes to zero. Of course, since
the Navier-Stokes-$\alpha$ model is well-posed, its solution semigroup induces a
measure in the trajectory space, starting from any initial measure in the phase
space, in such a way that the time-projections of such a measure form indeed a
statistical solution in the sense we consider here (see Section 2.6). It is
expected
that the converse is also true, namely, that any statistical solution for the
Navier-Stokes-$\alpha$ model is the family of projections of a measure induced
by the solution semigroup (as it happens for the two-dimensional Navier-Stokes
equations), but we do not address this issue here.

This is one of a few results about convergence of statistical solutions (see
e.g.
\cite{cw1997,cr2007,RT} for other equations) and it is the first result of
convergence for 
this type of statistical solution, and we believe the main ideas presented here
will be  valuable
for extending the result to other types of approximations (see \cite{BMR}). 

The structure of the paper is as follows. We start with some usual
definitions and basic results concerning the functional-analytic framework. In
Section
\ref{compactness}, we review some facts regarding Borel measures and we state
some compactness results in the space of Borel measures which may not be so
familiar to the reader and which will be used to obtain the convergence of the
measures related with the statistical solutions. Later, in Sections \ref{sec1.3}
and \ref{sec1.4}, we briefly introduce the Navier-Stokes equations and the
Navier-Stokes-$\alpha$ model, focusing on the results that will be
important throughout this paper. In Sections \ref{sec1.5} and \ref{sec1.6},
we recall the definition of statistical solutions for the
Navier-Stokes equations and introduce a notion of statistical solution
for the Navier-Stokes-$\alpha$ model.

The main goals of this paper are presented in the final Section \ref{sec2},
where
we state and prove the convergence theorems for the measures on the
trajectory space (Theorem \ref{mtheo6} and  Corollary
\ref{maincortimedepss}) and for the statistical solutions (Theorem
\ref{mtheo6nv},Corollary \ref{maincortimedepssnv} and Theorem
\ref{mainthmstatss}). This section is divided
into two parts, one for the time-dependent case and the other for the particular
case of stationary statistical solutions. The main condition for the convergence
of a family of statistical solutions of the
$\alpha$-Navier-Stokes equations to a statistical solution of the Navier-Stokes
equations as $\alpha$ goes to zero is that the mean kinetic energy of the family
be uniformly bounded. This provides the tightness of the family of measures that
guarantees the compactness of the family according to the theory of Topsoe
discussed in Section \ref{compactness}. The main work, then, is to show that
this compactness is sufficient to guarantee that the limit family of measures is
a statistical solution of the Navier-Stokes equations. A crucial step in this
proof is to show that the limit measure is carried by the space of Leray-Hopf
weak solutions of the Navier-Stokes equations. This step depends very much on a
result by Vishik, Titi, and Chepyzhov \cite{VTC} on the convergence of
individual solutions. 

\section{Preliminaries}\label{prem}

In this section, we set the notation and provide the definitions and results
that are needed throughout this work. 

\subsection{Functional setting}\label{sec1.1}

Let $\Omega:=\Pi_{i=1}^3(0, L_i)$, where $L_i> 0$, for $i = 1,2,3$, and let
$C^{\infty}_{per}(\Omega;\mathbb R^3)$ represent the space of the infinitely
differentiable functions $\bu$, with
$\bu(\x)=(u_1(\x),u_2(\x),u_3(\x))\in \mathbb R^3$ defined for
$\x=(x_1,x_2,x_3)\in \mathbb R^3$, which are $\Omega$-periodic.
We define the set of periodic test functions with vanishing average as
\[\mathcal V:=\{\bu\in \mathcal C^{\infty}_{per}(\Omega;\mathbb R^3):\nabla
\cdot \bu=0\mbox{ and }\int_\Omega \bu(\x)\rd\x=0\}.\]

Let $H$ be the closure of $\mathcal V$ in $L^2(\Omega;\mathbb R^3)$
and let $V$ be the closure of $\mathcal V$ in $H^1(\Omega;\mathbb R^3)$. 
The inner product and the norm in $H$ are defined, respectively, by
\[(\bu,\bv):=\int_{\Omega}\bu\cdot \bv \rd\x \;\mbox{ and } \quad |\bu|:
=\sqrt{(\bu,\bu)},\]
where $\bu\cdot\bv = \sum_{i=1}^3 u_iv_i$, and in $V$ they are defined by
\[(\!(\bu, \bv)\!):=(\nabla \bu,\nabla \bv) \;\mbox{ and }
 \quad \|\bu\|:=\sqrt{(\!(\bu,\bu)\!)},\]
where it is understood that $\nabla \bu=(\partial u_i/\partial x_j)_{i,j=1}^
3$ and that the second term is the integral of the componentwise product between
$\nabla \bu$ and $\nabla \bv$. We also consider the space $H$ endowed with its
weak topology, which we
denote by $H_w$.

Let $A$ be the \textbf{Stokes operator} defined as $A=-\mathbb P \Delta$, where
$\mathbb P: L^2(\Omega)^3\rightarrow H$ is the Leray-Helmholtz projection,
i.e., the orthogonal projector in  $L^2(\Omega)^3$ onto the subspace of
divergence-free vector fields. We
denote by $D(A)$  the domain of $A$, which is defined as the set of functions
$\bu \in V$ such that $A\bu\in H$. Recall that, in the periodic case with zero space
average, $A\bu=-\Delta \bu$ for $\bu \in D(A)= V\cap H^2(\Omega)^3$ and $A$ is a
positive self-adjoint operator with compact inverse, so that it has a sequence
$\{\lambda_i\}_{i\in \mathbb N}$ of
positive eigenvalues counted according to their multiplicity, in increasing
order, associated with an orthonormal basis $\{\bw_i\}_{i\in
\mathbb N}$ in $H$. Furthermore, the Poincar\'e inequality holds, i.e.,
for all $\bu\in
V$,
\begin{eqnarray}\label{poincare}
\lambda_1|\bu|^2\leq \|\bu\|^2,
\end{eqnarray}
where $\lambda_1>0$ is the first eigenvalue of the Stokes operator.

Denote by $P_k$ the \textbf{Galerkin projector} defined as the projector onto
the space spanned by the eigenfunctions associated with
the first $k$ eigenvalues, i.e.,
\[P_k\bu=\sum_{i=1}^k(\bu,\bw_i)\bw_i, \;\forall \bu\in H.\]
There are some important properties of $P_k$ that we want to highlight.
For instance, we have that, for all $\bu\in H$ and for all $k\in
\mathbb N$, $|P_k\bu|\leq |P_{k+1}\bu|$ and $|P_k\bu|\leq |\bu|$. We also
have that $P_k: H_w\rightarrow H$ is continuous.
Indeed, observe that the open sets in $H$ can be characterized by a basis of
neighborhoods
\[\mathcal O(\bu,r)=\{\bw\in H: |\bu-\bw|<r\},\]
for $\bu\in H$ and $r>0$. On the other hand, the collection of open sets in
$H_w$ has a characterization
by a basis of neighborhoods given by
\[\mathcal O_w(\bu,r,\bv_1,\ldots,\bv_N)=\{\bw\in H_w:
\sum_{i=1}^N|(\bu-\bw,\bv_i)|^2<r^2\},\]
for $\bu\in H_w$, $r>0$, $N\in \mathbb N$, and $\bv_1, \ldots, \bv_N\in H$.
Notice that if we prove that for all $\bu\in H_w$ and $r >0$ the set 
$P_k^{-1}(\mathcal O(P_k\bu,r))$ is open in $H_w$, then we obtain
that
$P_k:H_w\rightarrow H$ is continuous. This follows from the fact that 
$P_k^{-1}(\mathcal O(P_k\bu,r))$ is of the form 
\begin{equation*}
\begin{split} 
P_k^{-1}(\mathcal O(P_k\bu,r))&=\{\bw\in H_w: P_k\bw\in
\mathcal
O(P_k\bu,r)\}=\{\bw\in H_w: |P_k\bu-P_k\bw|<r\} \\& =\{\bw\in H_w:
\sum_{i=1}^k|(\bu-\bw,\bw_i)|^2<r^2\},
\end{split}
\end{equation*}
and, hence, is an element of the basis for $H_w$.

For the functional setting  concerning the Navier-Stokes-$\alpha$ model, we
adopt the framework introduced by Vishik, Titi and Chepyzhov in
\cite{VTC}.

We start with the natural space for the solutions of the 
Navier-Stokes-$\alpha$ model. Given an interval $I\subset \mathbb R$, we define
\begin{equation}\label{F}
\mathcal F_I=\{ \z : \z(\cdot) \in  L^2_{loc}(I;V)\cap
L_{loc}^{\infty}(I;H), \partial_t\z(\cdot) \in L^2_{loc}(I;D(A)')\}.
\end{equation}
We endow this space with its natural weak-type topology, which we term \textbf{$
\tau$ topology}, and which can be defined in terms of nets as follows: a net of
functions $\{\z_\gamma\}_\gamma
\subset \mathcal F_I$ converges to a function $\z \in \mathcal F_I$ in the
topology
$\tau$ if for each compact interval  $J\subset I$,
\[\z_\gamma \stackrel{*}{\rightharpoonup}\z \; \mbox{ in
}L^{\infty}(J;H), \quad \z_\gamma \rightharpoonup \z\;\mbox{ in }L^{2}(J;V),
\quad \mbox{and} \quad \partial_t\z_\gamma\rightharpoonup \partial_t \z \mbox{
in }L^{2}(J;D(A)').\]

Consider also the following Banach space 
\begin{equation}\label{Fb}
\mathcal F^b_I=\{ \z : \z(\cdot) \in  L^2_{b}(I;V)\cap L^{\infty}(I;H),
\partial_t\z(\cdot) \in L^2_{b}(I;D(A)')\},
\end{equation}
with norm given by
\[\|\z\|_{\mathcal
F_b}=\|\z\|_{L_b^2(I,V)}+\|\z\|_{L^\infty(I,H)}+\|\partial_t\z\|_{L_b^2(I,D(A)')
}, \]
where
\[\|\z\|_{L_b^2(I,V)}=\sup_{\{t\in I: t+1\in I\}}\int_t^{t+1}\|\z(s)\|^2ds,\]
and
\[\|\partial_t\z\|_{L_b^2(I,D(A)')}=\sup_{\{t\in I: t+1\in
I\}}\int_t^{t+1}\|\partial_t\z(s)\|_{ D(A)'}^2ds.\]

Finally, we introduce a natural space for the solutions of the Navier-Stokes
equations, which is the space $\mathcal C_{loc}(I,H_w)$ of continuous functions
from an interval $I\subset \mathbb R$ to $H_w$,
where, as defined above, $H_w$ stands for the space $H$ endowed with the weak topology. This space
can also be seen as the space of weakly
continuous function from  $I$ to $H$. The topology on $\mC_{loc}(I,H_w)$ is
that of uniform convergence in $H_w$ on compact intervals of $I$.

\subsection{Measures and Compactness on the space of
measures}\label{compactness}
The natural measure space in our framework is the space of Borel probability
measures on $\mC_{loc}(I,H_w)$. Since we are interested in the study of the
convergence of family of measures we need a compactness result on the space of
measures. For measures on a separable metrizable space there is the well-known
compactness result due to Prohorov \cite{prohorov}. In our case,
however, $\mC_{loc}(I,H_w)$ is not metrizable, so we use a more general
compactness result due to Topsoe (see \cite{Topsoe,Topsoebook,Topsoe2}), which is suitable for our framework.

Let $X$ be a Hausdorff space and let $\mB_X$ be the Borel $\sigma$-algebra on
$X$. The set of finite Borel measures on $X$ is denoted by $\mM(X)$ and the set
of Borel probability measures is denoted by $\mP(X)$. We say
that a measure $\mu$ is \textbf{tight} if for every set $A\in \mB_X$,
\[\mu(A) = \sup \{\mu(K) : K \mbox{ is compact}, \;K \subset
A\}.\]
Denote by $\mathcal M(X;t)$ the set of all Borel finite measures which are
tight, and by $\mathcal P(X;t)$ the set of all measures in $\mathcal
P(X)$ which are tight. If $X$ is a Polish space (i.e., a separable and
completely metrizable space)
then every finite Borel measure is tight, so that $\mM(X;t)=\mM(X)$ and
$\mP(X;t)=\mP(X)$.

We
say that a net $(\mu_\gamma)$ in $\mathcal P(X)$ is \textbf{uniformly tight} if
for every
$\varepsilon>0$ there exists a compact set $K\subset X$ such that
\[\mu_\gamma(X\setminus K)<\varepsilon, \quad \forall \gamma.\]

Let $Y$ be a Hausdorff space and let $F: X\rightarrow Y$ be a continuous
function. For a measure $\mu\in \mathcal M(X)$ we define the 
measure $F\mu$ induced by $\mu$ on $Y$ as $F\mu(E)=\mu(F^{-1}(E))$, for every
Borel set $E\subset Y$. In this setting, the Change of Variables Theorem
(see e.g. \cite{AB}) says that if $\varphi: Y\rightarrow \mathbb R$ is a
$F\mu$-integrable function then $\varphi\circ F$ is
$\mu$-integrable and   
\begin{equation}\label{induced-measure}
\int_X \varphi(F(x))d\mu(x)=\int_Y \varphi(y)dF\mu(y).
\end{equation}

Now, for the result of compactness on the space of
measures we endow the space $\mathcal M(X)$ with the weakest topology for
which the mapping $\mu\mapsto \mu(f)$ is
upper semicontinuous for every $f$ bounded, real-valued, upper semicontinuous
function on $X$, where $\mu(f)$ stands for the integral $\int_Xf(x)d\mu(x)$. 
We use the symbol $\wconv$
to denote the convergence of nets in $\mathcal M(X)$ with respect to this weak
topology. The spaces $\mathcal P(X)$, $\mathcal M(X;t)$, and $\mathcal P(X;t)$
are endowed with the topology inherited from $\mathcal M(X)$.

Recall that a topological space $X$ is \textbf{completely regular} if every
nonempty closed set and every singleton disjoint from it can be separated by a
continuous function.

In  \cite{Topsoe}, Topsoe proved a result of compactness on
the space of measures on an abstract space $X$. In the case $X$ is a topological
Haudorff space the result is reduced to the following(see \cite[Theorem 9.1]{Topsoebook}): 

\medskip
\begin{theorem}\label{reformulated}
Let $X$ be a Hausdorff space. Let $(\mu_\gamma)$ be a net in
$\mathcal P(X;t)$ which is uniformly tight. Then, there
exist $\mu\in \mathcal P(X;t)$ and a subnet  $(\mu_{\gamma_\beta})$ such that
$\mu_{\gamma_\beta}
\wconv\mu$.
\end{theorem}

Topsoe also proved, in \cite{Topsoebook}, the following result:

\medskip
\begin{lemma}\label{portmanteau}Let $X$ be a completely regular Hausdorff space.
For a net $(\mu_\gamma)$ in $\mathcal M(X)$ and $\mu\in \mathcal M(X, t)$,
the following statements are equivalent
\begin{itemize}
 \item [(1)]$\mu_\gamma \wconv\mu$;
 \item [(2)]$\limsup \mu_\gamma (f)\leq \mu(f)$, for all $f$ bounded upper
semicontinuous function;
 \item [(3)]$\liminf \mu_\gamma (f)\geq \mu(f)$, for all $f$ bounded lower
semicontinuous function;
 \item [(4)]$\lim\mu_\gamma(f)=\mu(f)$, for all bounded continuous 
function $f$.
\end{itemize}
\end{lemma}

\medskip
\begin{remark}
The last lemma was actually stated in a more general setting. For instance, if
$X$ is only
Hausdorff and $\mu\in \mathcal M(X)$ then the first three statements are
equivalent and each of them implies the last one. If $X$ is a completely regular
Hausdorff space and $\mu\in \mathcal M(X)$ is $\tau$-smooth (a condition which
is satisfied for every tight measure) then all the statements are equivalent. 
\end{remark}

Observe that, by Theorem \ref{reformulated},  if  $(\mu_\gamma)$ is a net in 
$\mathcal P(X;t)$ which is uniformly tight then there exists a subnet that
converges to a  limit $\mu$ which is tight. Moreover, by Lemma
\ref{portmanteau}, if $X$ is a completely regular Hausdorff space, the
convergence in the weak topology in $\mathcal P(X;t)$ is equivalent to the usual
convergence  $\mu_{\gamma_\beta}(f)\rightarrow
\mu(f)$, for every $f\in \mathcal C_b(X)$, denoted by $\mu_{\gamma_\beta}\wsconv
\mu$, where $\mathcal C_b(X)$ denotes the
set of bounded, continuous, real-valued functions defined on $X$.

We state the last observation as the next theorem:

\medskip
\begin{theorem}\label{part-topsoe}
Let $X$ be a completely regular Hausdorff space. Let $(\mu_\gamma)$ be a net in
$\mathcal P(X;t)$ which is uniformly tight. Then, there exist $\mu\in \mathcal
P(X;t)$ and
a subnet  $(\mu_{\gamma_\beta})$ such that $\mu_{\gamma_\beta}\wsconv
\mu$, i.e.,
\[\lim_\beta\mu_{\gamma_\beta}(f)=\mu(f), \quad\mbox{ for all  }f\in \mathcal
C_b(X).\]
\end{theorem}

Also in \cite{Topsoebook}, Topsoe stated the following result, which we prove
here with more details.

\medskip
\begin{theorem}
Let $X$ be a Hausdorff space. Then, $\mathcal P(X;t)$ is a Hausdorff space. 
\end{theorem}

\begin{proof}
First, recall that a Hausdorff space can be characterized as a topological
space where every net converges to at most one point. Therefore, it is enough 
to prove that if $(\mu_\gamma)_\gamma$ is a net in $\mathcal P(X;t)$ which
converges to two measures $\mu_1, \mu_2 \in \mathcal P(X;t)$,
i.e., $\mu_\gamma \wconv\mu_1$ and $\mu_\gamma \wconv\mu_2$, then
$\mu_1=\mu_2$. Let $A\in \mathcal B_X$, denote by $\mathring A$ the interior
of $A$ and by $\bar A$ the closure of $A$. It is clear that the
characteristic functions $\chi_{\mathring A}$ and $\chi_{\bar A}$ are,
respectively, lower semicontinuous and upper semicontinuous functions.
Therefore, using Lemma \ref{portmanteau}, we obtain that
\[\mu_1(\mathring A)\leq \liminf_\gamma \mu_\gamma(\mathring A)\leq
\limsup_\gamma \mu_\gamma(\bar A)\leq \mu_2(\bar A).\]
Now, let $E\in \mathcal B_X$, and let us prove that $\mu_1(E)\leq \mu_2(E)$. In
order to do so, consider any compact sets $K_1\subset E$ and $K_2\subset E^c$.
Since $X$ is Hausdorff there exist disjoint open sets $A$ and $B$ such that
such that $K_1\subset A$ and $K_2\subset B$. It is clear that $\bar A\subset
X\setminus K_2$. Thus,
\[\mu_1(K_1)\leq \mu_1(A)\leq \mu_2(\bar A)\leq
\mu_2(X\setminus K_2)=1-\mu_2(K_2),\]
which leads us to 
\[\mu_1(K_1)+\mu_2(K_2)\leq 1.\]
Since $K_1$ and $K_2$ are arbitrary compact sets satisfying $K_1\subset E$
and $K_2\subset E^c$, we can take the supremum over all compact
sets $K_1\subset E$ and the supremum  over all compact sets $K_2\subset
E^c$ in the last expression, and we find that 
\begin{equation*}
\sup\{\mu_1(K_1): K_1 \mbox{ is compact}, K_1\subset E \}
+\sup\{\mu_1(K_2):K_2 \mbox{ is compact}, K_2\subset E^c \} \leq 1.
\end{equation*}
Since $\mu_1$ and $\mu_2$ are tight, we conclude that
\[\mu_1(E)+\mu_2(E^c)\leq 1.\]
Thus $\mu_1(E)\leq \mu_2(E)$, for all $E\in \mathcal B_X$. Now, since
$\mu_1(X)=\mu_2(X)=1$ it follows that
$\mu_1=\mu_2$.
\end{proof}

Let $X$ be a completely regular Hausdorff space and $\mu_1, \mu_2\in \mathcal
P(X;t)$. Then, 
\begin{equation}\label{uniq-measu}
\mu_1=\mu_2 \mbox{ if and only if }\int_X\varphi(x)\rd
\mu_1(x)=\int_X\varphi(x)\rd \mu_2(x), \forall \varphi\in \mathcal C_b(X). 
\end{equation}
Indeed, suppose that
$\int_X\varphi(x)\rd
\mu_1(x)=\int_X\varphi(x)\rd \mu_2(x)$, 
for all $\varphi\in \mathcal C_b(X)$. We define the net
$(\mu_\gamma)_\gamma$, where $\mu_\gamma=\mu_2$,
for all $\gamma$. Then, it is clear that $\mu_\gamma \wconv\mu_1$ and
$\mu_\gamma \wconv\mu_2$. Since $\mathcal P(X;t)$ is
Hausdorff we find that $\mu_1=\mu_2$. The other implication is trivial.

\subsection{The Navier-Stokes equations}\label{sec1.3}

We state only the results, properties and estimates that are needed throughout
this work. For a more complete theory
of the Navier-Stokes equations the reader is referred to
\cite{bookcf1988,foiasholmtiti2001,Lady,temam84,temam1995},
and the references therein.

In the estimates below, we will consider certain quantities which are constant with respect to the solutions of the equations, but which may depend on the coefficients of the equations, the forcing terms and the spatial domain. We will call them non-dimensional constants when they are independent of re-scalings of the equations in space and time, hence, in particular, they may depend on the shape of the domain, but not on the size of the domain. Moreover, a non-dimensional constant will be called universal when it does not depend on any of the parameters of the equations, not even on the shape of the domain.

We recall that the Navier-Stokes equations can be written in the following
functional form
\begin{eqnarray}\label{NS}
\bu_t+\nu A\bu+B(\bu,\bu)=\f,
\end{eqnarray}
where $A$ is the Stokes operator and $B(\bu,\bu)=\mathbb P[(\bu\cdot \nabla)
\bu]$, for $\bu\in V$.

The notion of solution that is considered here is the well-known
Leray-Hopf weak solution, which is defined below.

\medskip
\begin{definition}\label{NSweak}Let $\f\in L_{loc}^2(I, V')$. A function
$\bu$ is called a \textbf{Leray-Hopf weak solution} if:
\begin{itemize}
 \item[(i)] $\bu\in L^{\infty}_{loc}(I, H)\cap L_{loc}^2(I, V)\cap \mathcal
C_{loc}(I,H_w)$;
\item[(ii)] $\partial_t\bu \in   L_{loc}^{4/3}(I, V')$;
 \item[(iii)]  $\bu$ satisfies the weak formulation of the Navier-Stokes
equations, i.e., 
\[\bu_t+\nu A\bu+B(\bu,\bu)=\f,\]
in $V'$,  in the sense of distributions on $I$;
 \item[(iv)] $\bu$ satisfies the energy inequality in the sense that for almost
all $t'\in I$ and for all $t\in I$ with $t>t'$,
\begin{eqnarray}\label{energy-ineq}
\frac{1}{2}
|\bu(t)|^2+\nu\int_{t'}^{t}\|\bu(s)\|^2\rd s \leq \frac{1}{2}
|\bu(t')|^2+
\int_{t'}^{t}(\f(s),\bu(s))\rd s;
\end{eqnarray}

 \item[(v)] If $I$ is closed and bounded on the left, with left end point
$t_0$, then the solution is strongly continuous in $H$ at $t_0$ from the right,
i.e., $\bu(t)\rightarrow \bu(t_0)$ in $H$ as $t \rightarrow t_0^+$.
\end{itemize}
\end{definition}

The set of allowed times $t'$ in \eqref{energy-ineq} can be characterized as 
the points of strong continuity of $\bu$, in $H$, from the right. In particular,
condition (v) implies that $t'=t_0$ is allowed in that case.

Suppose that $\f\in L^\infty(I,H)$ and let $\bu$ be a Leray-Hopf weak solution.
It is known that (see  e.g. \cite[Appendix II.B.1]{fmrt2001a}) 
\begin{eqnarray}\label{ener-est}
|\bu(t)|^2 \leq |\bu(t')|^2e^{-\lambda_1\nu
(t-t')}+\frac{1}{\lambda_1^2\nu^2}\|\f\|^2_{L^{\infty}(t',t;
H)}(1-e^{-\lambda_1\nu(t-t')}),
\end{eqnarray}
for all $t'\in I$ allowed in \eqref{energy-ineq} and for all $t\in I$ with $t>
t'$. 

Furthermore, for all $t'\in I$ allowed in \eqref{energy-ineq} and for all $t\in
I$ with $t> t'$,
\begin{equation}\label{nsest2}
\left(\int_{t'}^{t}\|\bu(s)\|^2ds\right)^{1/2}\leq
\frac{1}{\nu^{1/2}}|\bu(t')|+\lambda_1^{1/4}\nu M_1(t-t')^{1/2},
\end{equation}
and
\begin{equation}\label{nsest3}
\left(\int_{t'}^{t}\|\partial _t\bu(s)\|_{D(A)'}^2ds\right)^{1/2}\leq
\frac{c_1}{\lambda_1^{1/4}\nu^{1/2}}|\bu(t')|^2+\frac{\nu^{3/2}}{\lambda_1^{3/4}
}
M_1+ \nu^{5/2}\lambda_1^{1/4}M_1(t-t'),
\end{equation}
where $c_1$ is a universal constant and $M_1$ is a non-dimensional constant
which depends only on non-dimensional combinations of the parameters $\nu$,
$\lambda_1$ and $\|\f\|_{L^\infty(I,H)}$.

Define
\begin{eqnarray}\label{r0}
R_0:=\frac{1}{\lambda_1\nu}\|\f\|_{L^\infty(I,H)}
\end{eqnarray}
and observe that if $|\bu(t')|\leq R$, for some $R\geq R_0$ and some $t'\in I$
allowed in (\ref{energy-ineq}), then it follows from (\ref{ener-est})
that $|\bu(t)|\leq  R$ for all $t\geq t'$.

\subsection{The Navier-Stokes-$\alpha$ model}\label{sec1.4}

Most of the definitions, properties and results described below were
taken from \cite{VTC}. 

We consider the 3D Navier-Stokes-$\alpha$ model, in the periodic domain
$\Omega$:
\begin{eqnarray}\label{ans}\left\{\begin{array}{l}
\bv_t-\nu\Delta \bv +\bu\times (\nabla \times \bv)+\nabla p=\f,\\
\bv=\bu-\alpha^2\Delta\bu,\\
\nabla \cdot \bu=0,\end{array}\right.
\end{eqnarray}
where $\bu$ is the unknown (filtered) velocity field, $\bv$ is an auxiliary
variable, $p$ is the pressure, $\f$ is the external force, and $\alpha>0$ is a
constant.

A functional formulation of \eqref{ans} can be written as
\begin{equation}\label{NSalpha}
\frac{d}{dt}(\bu+\alpha^2A\bu)+\nu A(\bu+\alpha^2A\bu)+\tilde
B(\bu,\bu+\alpha^2A\bu)=\f,
\end{equation}
where $A$ is the Stokes operator, and $\tilde B(\bu,\bv):=\mathbb P[\bu\times
(\nabla \times \bv)]$ is defined for $\bu,\bv\in V$, with values in $V'$, and
which can be extended continuously to an operator from $(\bu,\bv)\in V\times H$
with values in $D(A)'$ (see \cite{FHT}).

The notion of solution  for the Navier-Stokes-$\alpha$
model that we consider is stated in the
next definition:

\medskip
\begin{definition}\label{def1}Let $\f\in L_{loc}^2(I,H)$. A
function $\bu$
is a solution of (\ref{ans}) on $I$ if
\begin{itemize}
 \item[(i)] $\bu\in L^{\infty}_{loc}(I;V)\cap L^2_{loc}(I;D(A))$;
 \item[(ii)] $\partial_t \bu\in L^2_{loc}(I;H)$;
 \item[(iii)] $\bu \in \mathcal C_{loc}(I;V)$;
 \item[(iv)] $\bu$ satisfies 
\[\dfrac{d}{dt}(\bu+\alpha ^2A\bu)+\nu A(\bu+\alpha ^2A\bu)+\tilde
B(\bu,\bu+\alpha ^2A\bu)=\f\]

 in $D(A)'$, in the sense of distributions
on $I$;
 \item[(v)] $\bu$ satisfies the energy equality in the sense that for all
$t',t\in I$ with $t>t'$,
\begin{equation}
\begin{split}
\dfrac{1}{2}(|\bu(t)|^2+\alpha^2\|\bu(t)\|^2)+&
\nu\int_{t'}^t(\|\bu(s)\|^2+\alpha^2|A\bu(s)|^2)ds \\&= \dfrac{1}{2}
(|\bu(t')|^2+\alpha^2\|\bu(t')\|^2)+\int_{t'}^t(\f(s),\bu(s))ds.
\end{split}
\end{equation}
\end{itemize}
\end{definition}
Observe that conditions $(ii)$, $(iii)$ and $(v)$ are consequences of $(i)$ and
$(iv)$.

The existence and uniqueness theorem of solution to the Navier-Stokes-$\alpha$
model was proved in \cite{FHT}: 

\medskip
\begin{theorem}Let $f\in H$ and $\bu_0\in V$. Then, for each $T>0$, there exists
a unique solution of (\ref{ans}) on $[0,T]$ in the sense of Definition
\ref{def1} with initial data $\bu_0$.
\end{theorem}

It is not hard to see that, in fact, one can assume $\f\in
L^\infty(I,H)$
and still have the same conclusion as in the previous theorem.

Another important property, which was observed and used in \cite{VTC} is the
following: 
If $\bu$ is a solution of the Navier-Stokes-$\alpha$ model on $I$, then $\bw$,
defined by $\bw=(1+\alpha^2 A)^{1/2}\bu$, satisfies the
 functional equation
\begin{equation}\label{ansnew}
\bw_t+\nu A\bw+(1+\alpha^2 A)^{-1/2}\tilde B((1+\alpha^2
A)^{-1/2}\bw,(1+\alpha^2 A)^{1/2}\bw)=(1+\alpha^2 A)^{-1/2}\f,
\end{equation}
and vice-versa. 

Therefore, based on the definition of solution to the equation (\ref{ans}), 
a notion of solution for the equation (\ref{ansnew}) can be defined as follows:

\medskip
\begin{definition}\label{def2}Let  $\f\in L_{loc}^2(I;H)$. A function $\bw$
is a solution of (\ref{ansnew}) on $I$ if:
\begin{itemize}
 \item[(i)] $\bw\in L^{\infty}_{loc}(I;H)\cap L^2_{loc}(I;V)$;
 \item[(ii)] $\partial_t \bw\in L^2_{loc}(I;D(A)')$;
 \item[(iii)] $\bw \in \mathcal C_{loc}(I;H)$;
 \item[(iv)] $\bw$ satisfies 
\[\bw_t+\nu A\bw+(1+\alpha^2 A)^{-1/2}\tilde B((1+\alpha^2
A)^{-1/2}\bw,(1+\alpha^2 A)^{1/2}\bw)= (1+\alpha^2 A)^{-1/2}\f,\]
in $D(A)'$, in the sense of distributions on $I$;
\item[(v)]$\bw$ satisfies the energy equality in the sense that for all
$t',t\in I$ with $t>t'$,
\begin{equation}\label{ener-eq}
\dfrac{1}{2}|\bw(t)|^2+
\nu\int_{t'}^t\|\bw(s)\|^2ds=\dfrac{1}{2}
|\bw(t')|^2+\int_{t'}^t((1+\alpha^2 A)^{-1/2}\f(s),\bw(s))ds.
\end{equation}
\end{itemize}
\end{definition}

Again, here we have that conditions $(ii)$, $(iii)$ and $(v)$ are consequences
of $(i)$ and $(iv)$.

Suppose that $\f\in L^{\infty}(I,H)$. We observe that if $\bw$ is a solution of
the Navier-Stokes-$\alpha$ model on an interval $I$, in the sense of Definition
\ref{def2}, then $|\bw(\cdot)|^2$ is an absolutely continuous function on $I$
(see \cite[Corollary 2.1]{VTC}). As a consequence, one can show that, for any
$\psi: [0,\infty)\rightarrow
\mathbb R$, such that $\psi\in \mC^1([0,\infty))$, $\psi\geq 0$, $\sup_{r\geq 0}
\psi'(r) <\infty$, $\bw$ satisfies the following estimate, 
\begin{equation}\label{streng-ener}
\psi(|\bw(t)|^2)\leq 
\psi(|\bw(t')|^2)+\frac{1}{\lambda_1\nu}\|\f\|_{L^\infty(I,
H)}\sup_{r\geq 0}\psi'(r)(t-t'),
\end{equation}
for all $t,t'\in I$, with $t>t'$.

We also have the following \textit{a~priori}
estimates for any solution $\bw$ of the Navier-Stokes-$\alpha$ model in the
sense of
Definition \ref{def2} (adapted from \cite[Corollary 3.2]{VTC}):

\begin{eqnarray}\label{est1}
|\bw(t)|^2\leq |\bw(t')|^2
e^{-\nu\lambda_1(t-t')}+\frac{1}{\lambda_1^2\nu^2}\|\f\|^2_{L^\infty(t',t;H)}
(1-e^{
-\nu\lambda_1(t-t')}),
\end{eqnarray}
\begin{eqnarray}\label{est2}
\left(\int_{t'}^{t}\|\bw(s)\|^2ds\right)^{1/2}\leq 
\frac{1}{\nu^{1/2}}|\bw(t')|+\lambda_1^{1/4}\nu M_2(t-t')^{1/2},
\end{eqnarray}
\begin{equation}\label{est3}
\left(\int_{t'}^{t}\|\partial _t\bw(s)\|_{D(A)'}^2ds\right)^{1/2}\leq
\frac{c_2}{\lambda_1^{1/4}\nu^{1/2}}|\bw(t')|^2+\frac{\nu^{3/2}}{\lambda_1^{3/4}
}
M_2+ \nu^{5/2}\lambda_1^{1/4}M_2(t-t'),
\end{equation}
for all $t', t\in I$, with $t>t'$, where $c_2$ is a universal constant and $M_2$
is a non-dimensional constant which depends only
on non-dimensional combinations of the parameters $\nu$, $\lambda_1$ and
$\|\f\|_{L^\infty(I,H)}$.

Again, we observe that if $|\bw(t')|\leq R$, for some $R\geq R_0$, where $R_0$
is given by (\ref{r0}), and for some $t'\in I$, then it follows from
(\ref{est1}) that $|\bw(t)|\leq R$ for all $t\geq t'$.

In \cite[Theorem 3.1]{VTC}, Vishik, Titi and Chepyzhov proved the convergence of
solutions
of the Navier-Stokes-$\alpha$ model to solutions of the Navier-Stokes
equations. This result was proved in the case where $f\in H$ and $I=[0,
\infty)$, which was of interest to them. However, it is not hard to see that
the proof can be adapted to the
case when $I$ is an arbitrary interval and $f\in L^\infty(I,H)$. Since this
result is going to play an important
role in this article and for the reader's convenience we state it below:

\medskip
\begin{theorem}\label{v-t-c}
Let $I\subset \mathbb{R}$ be an arbitrary interval and $f\in L^\infty(I,H)$. Let
$\{\bw_n\}$ be a bounded sequence in $\mathcal F^b_I$ such that each $\bw_n$
is a solution of the Navier-Stokes-$\alpha_n$ model on $I$, with
$\alpha_n\rightarrow 0$ as $n \rightarrow \infty$, and $\bw_n\rightarrow \bw$ in
$\tau$ as $n\rightarrow \infty$, for some $\bw\in\mathcal F^b_I$. Then $\bw$ is
a Leray-Hopf weak solution of the 3D
Navier-Stokes equations on $\mathring{I}$.
\end{theorem}

\subsection{Trajectory spaces}\label{sec1.5}

The trajectory spaces made of solutions of the Navier-Stokes equations and of the Navier-Stokes-$\alpha$ model play an important role in the definition of the
Vishik-Fursikov measure (see Definitions \ref{vfmdef} and \ref{vfalphamdef})
since they connect this notion of solution with the corresponding equations.

Let $I\subset \mathbb{R}$ be an arbitrary interval and $R>0$. 
Consider the spaces $\mathcal C_{loc}(I, H_w)$ and $\mathcal C_{loc}(I,
B_H(R)_w)$ endowed with the topology of uniform weak convergence on
compact intervals in $I$. Then,
$\mathcal C_{loc}(I, H_w)$ is a separable Hausdorff locally convex
topological vector space, hence completely regular, and $\mathcal C_{loc}(I,
B_H(R)_w)$ is a Polish space,
that
is, a separable and completely metrizable space. 

For any interval $J\subset I$, we introduce the restriction operator 
\begin{eqnarray*}
\Pi_J: \mathcal C_{loc}(I, H_w)&\rightarrow &\mathcal C_{loc}(J, H_w) \\
\bu&\mapsto & (\Pi_J\bu)(t)=\bu(t), \; \forall t\in J.
\end{eqnarray*}
It is clear that the restriction operator is continuous. Furthermore, if $J$
is a closed subinterval of $I$, then $\Pi_J$ is also surjective and open.

For each interval $I$ in $\mathbb R$ and each $t\in I$ we define the projection
operator $\Pi_t$ by
\begin{eqnarray*}
\Pi_t: \mathcal C_{loc}(I, H_w)&\rightarrow & H_w \\
\bu&\mapsto & \Pi_t\bu=\bu(t).
\end{eqnarray*}
which is also  continuous, surjective and open.

For the Navier-Stokes equations, we define the following trajectory spaces
based on the Leray-Hopf weak solutions given by Definition \ref{NSweak}:
\begin{equation}\label{u}
\mathcal U_I=\{\bu\in \mathcal C_{loc}(I;H_w): \bu \mbox{ is a Leray-Hopf
weak solution on } I\},
\end{equation}
\begin{equation}\label{u-sharp}
 \mathcal U^{\sharp}_I=\{\bu\in \mathcal C_{loc}(I;H_w): \bu \mbox{ is a
Leray-Hopf weak
solution on } \mathring{I}\},
\end{equation}
where $\mathring I$ represents  the interior of $I$. We endow these spaces with
the 
topology inherited from $\mathcal C_{loc}(I,H_w)$.

The relation between the two spaces defined above is that $\mathcal
U_I^{\sharp}$ is the sequential closure of $\mathcal
U_I$ with respect to the topology inherited from $\mathcal C_{loc}(I;H_w)$.
Furthermore it is clear that if $I$ is open then $ \mathcal U_I=
\mathcal U^{\sharp}_I$. The difference appears when $I$ is closed and bounded
on the left, since we do not know, in general, whether the weak solutions in
$\mathcal U^\sharp_I$ are strongly continuous in $H$, from the right, at the
left end point of
the interval, as are those in $\mathcal U_I$.

Sometimes it will be useful to work with the following spaces 
\begin{equation}\label{ur}
\mathcal U_I(R)=\{\bu\in \mathcal C_{loc}(I;B_H(R)_w): \bu \mbox{ is a
Leray-Hopf weak solution on } I\},
\end{equation}
\begin{equation}\label{u-sharp-r}
 \mathcal U^{\sharp}_I(R)=\{\bu\in \mathcal C_{loc}(I;B_H(R)_w): \bu \mbox{ is a
Leray-Hopf weak solution on } \mathring{I}\},
\end{equation}
with the topology inherited from $\mathcal C_{loc}(I;H_w)$.

As observed in \cite{FRT}, if $I$ is an interval closed and bounded on the
left, then, for any sequence $\{R_i\}_{i=1}^\infty$ of positive numbers with
$R_i\geq R_0$, for all $i\in\mathbb N$, and $R_i\rightarrow \infty$, we have the
representation
\begin{equation}\label{represUclosed}
\mathcal U_I^{\sharp}=\bigcup_{i=1}^\infty\mathcal
U^\sharp_I(R_i).
\end{equation}
And, if $I$ is an interval open on the left, then, for any sequence
$\{R_i\}_{i=1}^\infty$ of
positive numbers with $R_i\geq R_0$, for all $i\in\mathbb N$, and
$R_i\rightarrow \infty$ and for any
sequence $\{J_n\}_{n=1}^\infty$ of compact intervals in $I$ such that 
$\cup_{n=1}^\infty J_n=I$, we have the representation

\begin{equation}\label{represUopen}
\mathcal
U_I^{\sharp}=\bigcap_{n=1}^\infty\bigcup_{i=1}^\infty\Pi^{-1}_{J_n}\mathcal
U^\sharp_{J_n}(R_i).
\end{equation}

As proved in \cite{FRT}, the spaces $\mU_I$, $\mU_I^\sharp$, $\mU_I(R)$ and $\mU_I^\sharp(R)$ are all Borel subsets of $\mC_{loc}(I,H_w)$ and, in particular, $\mU_I^\sharp(R)$ is closed.

For the Navier-Stokes-$\alpha$ model, we consider the solutions in the sense of
Definition \ref{def2}. Since these solutions belong to
$\mathcal \mathcal C_{loc}(I;H)$, we define the following
trajectory spaces 
\begin{equation}\label{u-alpha}
 \mathcal U^{\alpha}_I=\{\bu\in \mathcal C_{loc}(I;H): \bu \mbox{ is a solution
of the Navier-Stokes-$\alpha$ model on } I\},
\end{equation} 
\begin{equation}\label{u-alpha-r}
\mathcal U^{\alpha}_I(R)=\{\bu\in \mathcal C_{loc}(I;B_H(R)): \bu  \mbox{ is a
solution of the Navier-Stokes-$\alpha$ model on } I\}.
\end{equation}
In order to compare with the solutions of the Navier-Stokes equations and since
$\mathcal C_{loc}(I,H)$ is included in $\mathcal C_{loc}(I,H_w)$, we shall consider the
spaces $\mathcal U^{\alpha}_I$ and $\mathcal U^{\alpha}_I(R)$ with the topology
inherited from $\mathcal C_{loc}(I,H_w)$. 

Here, we also have the same characterizations as the ones for the Navier-Stokes
trajectory space. That is,  if $I$ is an interval closed and bounded on the
left, then, for any sequence $\{R_i\}_{i=1}^\infty$ of positive numbers with
$R_i\geq R_0$, for all $i\in\mathbb N$, and $R_i\rightarrow \infty$, we have the
representation
\begin{equation}\label{represUalphaclosed}
\mathcal U_I^{\alpha}=\bigcup_{i=1}^\infty\mathcal
U^\alpha_I(R_i).
\end{equation}
And, if $I$ is an interval open on the left then, for 
any sequence $\{R_i\}_{i=1}^\infty$ of
positive numbers with $R_i\geq R_0$, for all $i\in\mathbb N$, and
$R_i\rightarrow \infty$ and for any
sequence $\{J_n\}_{n=1}^\infty$ of compact intervals in $I$ such that 
$\cup_{n=1}^\infty J_n=I$, we have that

\begin{equation}\label{represUalphaopen}
\mathcal
U_I^{\alpha}=\bigcap_{n=1}^\infty\bigcup_{i=1}^\infty\Pi^{-1}_{J_n}\mathcal
U^\alpha_{J_n}(R_i).
\end{equation}

The space $\mathcal U^{\alpha}_I(R)$ is closed in the topology inherited from $\mathcal C_{loc}(I,H_w)$ (and this implies that it is also closed in $\mathcal C_{loc}(I,H)$ since $H$ is continuously included in $H_w$). Indeed, since the solutions in $\mathcal U^{\alpha}_I(R)$ are uniformly bounded by $R$ in $H$, it suffices to show that  $\mathcal U^{\alpha}_I(R)$ is closed in $\mathcal C_{loc}(I,B_H(R)_w)$. Since $\mathcal C_{loc}(I,B_H(R)_w)$ is metrizable, it suffices to work with sequences. Then, if $\{\bu_n\}_n$ is a sequence in $\mathcal U^{\alpha}_I(R)$ which converges in $\mathcal C_{loc}(I,B_H(R)_w)$ to an element $\bu$, then the \textit{a~priori} estimates \eqref{est1}, \eqref{est2}, \eqref{est3} yield the compactness of this sequence in $\mathcal C_{loc}(I,H_w)$ and in $\mathcal F_I$. This compactness is sufficient
to show that the limit function $\bu$ is a solution of the Navier-Stokes-$\alpha$ model and, hence, belongs to $\mathcal U^\alpha_I(R)$, proving that this space is closed. Since this space is closed, there is no need to consider spaces analogous to $\mathcal U^\sharp_I(R)$ and $\mathcal U^\sharp_I$, as done for the Navier-Stokes equations.

Due to the representations \eqref{represUalphaclosed} and \eqref{represUalphaopen}, we see that $\mU_I^\alpha$ is an $\mF_\sigma$-set in $\mathcal C_{loc}(I,H_w)$, in the case $I$ is closed and bounded on the left, and it is an $\mF_{\sigma\delta}$-set, in the case $I$ is open on the left. In any case, $\mU_I^\alpha$ is a Borel set.

We now introduce an auxiliary functional space, $\mY_I$, which is directly
connected with
the \emph{a~priori} estimates for the solutions of the Navier-Stokes equations
and of the Navier-Stokes-$\alpha$ model, with suitable compactness property.
First, let $J$
be a compact interval
in $\mathbb R$, then we define 
\begin{multline}\label{y}
\mathcal Y_{J}(R)= \left\{\bu\in \mathcal C_{loc}(J;H_w): 
 |\bu(t)|\leq R, 
   \|\bu\|_{L^2(s,t;V)}\leq
\frac{1}{\nu^{1/2}}R+\lambda_1^{1/4} \nu M (t-s)^{1/2}, \mbox{ and}\right.\\
  \left. \|\partial_t\bu\|_{L^{2}(s,t;D(A)')}\leq
\frac{c}{\lambda_1^{1/4}\nu^{1/2}}R^2+\frac{\nu^{3/2}}{\lambda_1^{3/4}}M 
+ \nu^{5/2}\lambda_1^{1/4}M(t-s), \forall s,t\in J\right\},
\end{multline}
where $c=\max\{c_1,c_2\}$ is a universal constant and $M=\max\{M_1,M_2\}$ is a
non-dimensional constant which depends only on non-dimensional combinations of
the terms $\nu$, $\lambda_1$ and $\|\f\|_{L^\infty(I,H)}$. With these choices of
constants, for any $R\geq R_0$, if $|u_0| \leq R$, then the solutions of the
Navier-Stokes equations and of the Navier-Stokes-$\alpha$ model with initial
condition $u_0$ all satisfy the estimates in \eqref{y}, for subsequent times,
which is possible thanks to the a~priori estimates \eqref{ener-est},
\eqref{nsest2} and
\eqref{nsest3}, and \eqref{est1}, \eqref{est2}, and \eqref{est3}. Thus, we have
\begin{equation}
  \label{ujinyj}
  \mU_J(R), \;\mU_J^\alpha(R) \subset \mY_J(R), \quad \forall  R\geq R_0.
\end{equation}

Now, let $I$ be any interval in $\mathbb R$ and $\{J_n\}_{n=1}^\infty$ be a
sequence of compact intervals in $I$ such that $J_n\subset J_{n+1}$ and
$\cup_{n=1}^\infty J_n=I$. Consider also a sequence  $\{R_i\}_{i=1}^\infty$ of
increasing  real number such that $R_1\geq R_0$ and $R_i\rightarrow
\infty$ as $i\rightarrow \infty$. Define 
\begin{equation}\label{represY}
\mathcal Y_I:=\bigcap_{n=1}^\infty \bigcup_{i=1}^\infty
\Pi^{-1}_{J_n}
\mathcal Y_{J_n}(R_i).
\end{equation}
Also, for a given $R\geq R_0$, we define
\begin{eqnarray}\label{yI}
\mathcal Y_{I}(R)=\bigcap_{n=1}^\infty 
\Pi^{-1}_{J_n}
\mathcal Y_{J_n}(R).
\end{eqnarray}
Observe that $\bigcup_{i=1}^\infty\mathcal Y_{I}(R_i)\subset \mathcal
Y_I\subset \mathcal C_{loc}(I,H_w)$. The space $\mY_I(R)$ is independent of the
choice of the intervals $\{J_n\}_n$, while the space $\mY_I$ is independent of
the choice of both the intervals $\{J_n\}_n$ and the sequence $\{R_i\}_i$,
although these properties are not really necessary.

\medskip
\begin{lemma}\label{ycompact}
Let $J$ be a compact interval in $\mathbb R$. Then,  $\mathcal Y_J(R)$ is a
compact subset of $\mathcal C_{loc}(J,H_w)$.
\end{lemma}

\begin{proof}
First, observe that since $\mathcal Y_J(R)\subset \mathcal C_{loc}(J, B_H(R)_w)$
then
$\mathcal Y_J(R)$ is metrizable. Now, let $\{\bu_n\}_n$ be a sequence in
$\mathcal Y_J(R)$. It is clear that $\{\bu_n\}$ is bounded in $L^2(J,V)$ and 
$\{\partial_t \bu_n\}$ is bounded in
$L^{2}(J,D(A)')$. Then, by Aubin's Compactness Theorem, we obtain that
$\{\bu_n\}$ is
relatively compact in $L^2(J,H)$. Using all
the information obtained before, we conclude that there exist a vector field
$\bu$ and a subsequence $\{\bu_{n_k}\}$ such that 
\[\begin{array}{ccc}
\bu_{n_k} \stackrel{*}{\rightharpoonup}\bu \; \mbox{ in
}L^{\infty}(J;H);\\
 \bu_{n_k} \rightharpoonup \bu\;\mbox{ in }L^{2}(J;V); \\
 \partial_t\bu_{n_k}\rightharpoonup \partial_t\bu\;\mbox{ in }L^{2}(J;D(A)');\\
 \bu_{n_k}\rightarrow \bu\;\mbox{ in }L^{2}(J;H).
\end{array}\]

Now, consider $\{\bw_i\}_i$ a countable dense subset in $D(A)$ (the existence
of such a set follows from the fact that $D(A)$ is separable). For each
$i\in \mathbb N$, define the sequence $\{f^i_k\}_k$ such that, for each $k\in
\mathbb N$, $f^i_k(t)=(\bu_{n_k}(t),\bw_i)$ for all $t\in J$. Note that
$\{f^i_k\}_k$ is a sequence of continuous functions from $J$ to $\mathbb R$,
which is uniformly bounded and 
equicontinuous. Indeed, since 
\[|f^i_k(t)-f^i_k(s)|=\left|\int_{t}^{s}
(\partial_\tau \bu_{n_k}(\tau),\bw_i)\rd \tau\right|\leq
|t-s|^{1/2}\|\bw_i\|_{D(A)} \|\partial_t \bu_{n_k}\|_{L^{2}(t,s;D(A)')},\]
for all $t,s\in J$, then $\{f^i_k\}_k$ is equicontinuous. Thus, we can apply
Arzel\`a-Ascoli Theorem to obtain that $\{f^i_k\}_k$ is relatively compact in
$\mathcal C_{loc}(J,\mathbb R)$, for each $i\in \mathbb N$. By Cantor's diagonal
argument, we can construct a subsequence $\{f^i_{k_j}\}$ such that $f^i_{k_j}$
converges to $(\bu(\cdot),\bw^i)$ as $j\rightarrow \infty$,  for all $i\in
\mathbb N$. Now we use a triangulation argument to
obtain the convergence of $\bu_{n_{k_j}}\rightarrow \bu$ in $\mathcal
C_{loc}(J,H_w)$.
More precisely, given $\varepsilon>0$ and $\bv \in H$, there exist $\bw\in D(A)$
such that $|\bv-\bw|<\varepsilon/(6R)$, and $i_0\in \mathbb N$ such that
$\|\bw-\bw_{i_0}\|_{D(A)}<\lambda_1\varepsilon/(6R)$. By the
convergence of $\{f^i_{k_j}\}_j$ we conclude that there exists $N\in \mathbb N$
such that $\sup_{t\in
J}|(\bu_{n_{k_j}}(t)-\bu(t),\bw_i)|\leq \varepsilon/3$, for all $j\geq N$ and
for all $i\in \mathbb N$.
Therefore, for all $j\geq N$, we have that
\begin{equation*}
\begin{split}
&\sup_{t\in J}|(\bu_{n_{k_j}}(t)-\bu(t),\bv)|\leq  \sup_{t\in
J}|(\bu_{n_{k_j}}(t)-\bu(t),\bv-\bw)|\\&+\sup_{t\in
J}|(\bu_{n_{k_j}}(t)-\bu(t),\bw-\bw_{i_0})|+\sup_{t\in
J}|(\bu_{n_{k_j}}(t)-\bu(t),\bw_{i_0})|\\ &<\sup_{t\in
J}|\bu_{n_{k_j}}(t)-\bu(t)|(|\bv-\bw|+|\bw-\bw_{i_0}|)+\frac{\varepsilon}{3}\\
&\leq
2R|\bv-\bw|+2R\frac{1}{\lambda_1}\|\bw-\bw_{i_0}\|_{D(A)}+\frac{\varepsilon}
{2}< \varepsilon.
\end{split}
\end{equation*}
Finally, it is clear that $\bu$ inherits the uniform estimates of the
sequence in $\mathcal Y_J(R)$, so that $\bu$ itself is in $\mathcal Y_J(R)$,
completing the proof that $\mathcal Y_J(R)$ is
compact. 
\end{proof}

\medskip
\begin{lemma}
Let $I$ be any interval in $\mathbb R$ and $R>0$. Then, the set $\mathcal
Y_{I}(R)$ is compact in $\mathcal C_{loc}(I,H_w)$. Moreover, $\mathcal Y_I(R)$
is metrizable. 
\end{lemma}

\begin{proof}
As observed in the end of Section 2.4 in \cite{FRT}, compact subsets of
$\mathcal C_{loc}(I,H_w)$ can be characterized as the sets $K$ for which, for
every compact interval $J \subset I$, the subset $\Pi_J K$ is equi-bounded with
respect to the norm of $H$ and  equicontinuous with respect to the uniform
structure of $\mathcal C_{loc}(J, H_w)$. Since $\Pi_J\mathcal Y_I(R)\subset
\mathcal
Y_J(R)$
and $\mathcal Y_J(R)$ is compact as proved in Lemma \ref{ycompact}, these
conditions are met, and we have that $\mathcal Y_I(R)$ is compact. Furthermore,
since $\mathcal Y_I(R)\subset
\mathcal C_{loc}(I,B_H(R)_w)$ then $\mathcal Y_I(R)$ is metrizable.
\end{proof}

\medskip
\begin{lemma}
\label{spaceyi}
Let $I$ be any interval in $\mathbb R$. Then, the space $\mathcal Y_I$, endowed
with the topology inherited from $\mathcal
C_{loc}(I,H_w)$, is a completely regular topological space. Moreover, $\mathcal
Y_I$
contains the spaces $\mathcal U_I^\sharp$ and $\mathcal U_I^{\alpha}$.
\end{lemma}

\begin{proof}
Since $\mathcal Y_I$ is a subspace of $\mathcal C_{loc}(I,H_w)$ then  $\mathcal
Y_I$ is a completely regular topological space. For the inclusions,
suppose first that $I$ is an
interval closed and bounded on the left. Then, for every compact interval
$J\subset I$ containing the left end point of $I$, and for all $R\geq
R_0$, where $R_0$ is defined by (\ref{r0}), it follows from \eqref{ujinyj} that $\mathcal U_I^\sharp(R) \subset \Pi_J^{-1}\mathcal Y_J(R)$.
Therefore, using (\ref{yI}) and (\ref{represUclosed}) we conclude that $\mathcal
U_I^\sharp \subset \mathcal Y_I$.  Now, in order to prove that $\mathcal
U^{\alpha}_I\subset \mathcal Y_I$, notice that from \eqref{ujinyj} we have that $\mathcal
U^{\alpha}_I(R)\subset \Pi_J^{-1}\mathcal Y_I(R)$, for all $R\geq
R_0$ and any compact interval  $J\subset I$  containing the left end point of
$I$. Thus, using (\ref{yI}) and (\ref{represUalphaclosed}) we conclude that
$\mathcal U_I^\alpha \subset \mathcal Y_I$.

Now, suppose that $I$ is open on the left. Again, we have from \eqref{ujinyj} that $\mathcal
U^\sharp_J(R)\subset \mathcal Y_J(R)$, for all $R\geq
R_0$ and any compact interval  $J\subset I$. Thus, by (\ref{represY}) and
(\ref{represUopen}), we conclude that $\mathcal
U^\sharp_I \subset \mathcal Y_I$. Observe also that we have from \eqref{ujinyj} that $\mathcal
U^{\alpha}_J(R)\subset \mathcal Y_J(R)$, for all $R\geq
R_0$ and any compact interval  $J\subset I$. Thus,
(\ref{represY}) and (\ref{represUalphaopen}) imply that $\mathcal
U^{\alpha}_I\subset \mathcal Y_I$.
\end{proof}

\medskip
\begin{remark}
If $I$ is an interval closed and bounded on the left and 
$\{R_i\}_i$ is a sequence of positive numbers with $R_i\geq R_0$, for all
$i\in \mathbb N$, and
$R_i\rightarrow \infty$, then we have in fact showed in the proof of Lemma
\ref{spaceyi} that $\mathcal U^\alpha_I, \mathcal U^\sharp_I\subset
\bigcup_{i=1}^\infty\mathcal
Y_I(R_i)$. On the other hand, if $I$ is an interval open on the left then
$\mathcal
U^\alpha_I, \mathcal U^\sharp_I\subset \mathcal Y_I$ but $\mathcal
U^\alpha_I$ and $\mathcal U^\sharp_I$ might be not included
in $\bigcup_{i=1}^\infty\mathcal
Y_I(R_i)$.
\end{remark}

\medskip
\begin{lemma}
Let $I$ be any interval in $\mathbb R$. Then, the space
$\mathcal Y_I(R)$ contains the spaces $\mathcal U_I^\sharp(R)$ and $\mathcal
U_I^{\alpha}(R)$,  for all $R\geq R_0$, where $R_0$ is defined by (\ref{r0}).
\end{lemma}
\begin{proof}
It is clear from estimates (\ref{ener-est}), (\ref{nsest2}) and
(\ref{nsest3}) that,  if $\bu\in \mU^\sharp(R)$, for some $R\geq R_0$, then
$\bu\in \Pi_{J}^{-1}\mY_I(R)$, for all compact interval $J\subset I$. Thus, it
follows from (\ref{yI}) that $\mathcal U_I^\sharp(R)\subset \mathcal Y_I(R)$.
With an analogous argument we can also prove that $\mathcal
U_I^\alpha(R)\subset \mathcal Y_I(R)$, for all $R\geq R_0$.
\end{proof}

Next, we prove an important convergence result concerning the trajectory spaces
$\mathcal U_I^{\sharp}(R)$ and $\mathcal U^{\alpha}_I(R)$, based on
Theorem \ref{v-t-c}.

\medskip

\begin{lemma}\label{dist}
Let $R\geq R_0$, where $R_0$ is defined by (\ref{r0}). Let $\mathcal U_I^{\sharp}(R)$ and $\mathcal
U^{\alpha}_I(R)$ be given by (\ref{u-sharp-r}) and (\ref{u-alpha-r}),
respectively. Then,
\[\lim_{\alpha\rightarrow 0}\dist_{\mathcal Y_I(R)}(\mathcal U_I^{\alpha}(R),
\mathcal U_I^{\sharp}(R))=0\]
where \[\displaystyle \dist_{\mathcal Y_I(R)}(\mathcal U_I^{\alpha}(R),
\mathcal
U_I^{\sharp}(R))=\sup_{\bw\in \mathcal U_I^{\alpha}(R)}\;\inf_{\bu\in \mathcal
U_I^{\sharp}(R)}d(\bw,\bu),\] and $d$ is any compatible metric in $\mathcal
Y_I(R)$.
\end{lemma}

\begin{proof}
Suppose by contradiction that \[\lim_{\alpha\rightarrow 0}\dist_{\mathcal
Y_I(R)}(\mathcal U_I^{\alpha}(R),
\mathcal U_I^{\sharp}(R))\neq 0.\]
Thus, there exist $\varepsilon>0$ and a sequence of $\{\alpha_n\}_n$, with $\alpha_n\rightarrow 0$ as $n\rightarrow \infty$, such that 
\[\sup_{\bw\in \mathcal U_I^{\alpha_n}(R)}\;\inf_{\bu\in \mathcal
U_I^{\sharp}(R)}d(\bw,\bu)>\varepsilon, \;\forall n \in \mathbb N.\]
Observe that, from the definition of the supremum, we have that, 
given $r>0$,  there exists $\bw_n\in \mathcal
U_I^{\alpha_n}(R)$ such that 
\[\inf_{\bu\in\mathcal U_I^{\sharp}(R)}d(\bw_n, \bu)>\varepsilon -r, \;\forall
n\in
\mathbb N. \]
In particular, we can take $r=\varepsilon/2$ and obtain 
\[\inf_{\bu\in\mathcal U_I^{\sharp}(R)}d(\bw_n, \bu)>\frac{\varepsilon}{2},
\;\forall n\in\mathbb N, \]
so that 
\begin{eqnarray}\label{ineq}
d(\bw_n, \bu)>\frac{\varepsilon}{2}, \;\forall
n\in\mathbb N \mbox{ and } \forall \bu \in \mathcal U_I^{\sharp}(R).
\end{eqnarray}
On the other hand, we have that $|\bw_n(t)|_H\leq R$ for all $t\in I$ and for
all $n\in\mathbb N$. Thus, estimates (\ref{est2}) and
(\ref{est3}) imply that $\{\bw_{n}\}_n$ is bounded in $\mathcal F^b_I$ (see
(\ref{Fb})).
Therefore, there exist a subsequence $\{\bw_{n_l}\}_l$ of $\{\bw_{n}\}_n$ and
a function $\bu\in \mathcal F^b_I$ such that $\bw_{n_l}\rightarrow
\bu$ with respect
to $\tau$. Using Theorem \ref{v-t-c}, we conclude that $\bu \in \mathcal
U_I^{\sharp}(R)$.
Moreover, for any compact interval $J\subset I$, we have that $\{\bw_{n}\}_n$ is
in $\mathcal Y_J(R)$. Then using Lemma \ref{ycompact} we
conclude that there exists a subsequence of  $\{\bw_{n}\}_n$ that converges 
in the topology of weak converge in $H$ uniformly on $J$. We can now use
Cantor's diagonal argument to obtain a subsequence $\{\bw_{n_k}\}_{n_k}$ that
converges to $\bu$ in the topology of weak converge in $H$ uniformly on $J$, for
any compact interval $J\subset I$. Observe that this lead us to a
contradiction with (\ref{ineq}) since $d$ is a compatible metric with the
topology of $\mathcal Y_I(R)$. 
\end{proof}

\subsection{Statistical solutions}\label{sec1.6}

The notion of statistical solutions that is considered here was introduced by
Foias, Rosa and Temam, in  \cite{frtcr2010a,frtcr2010b} (see
also \cite{FRT,frtssp2}). We recall this definition in the context of the
Navier-Stokes equations and introduce a corresponding definition for the
Navier-Stokes-$\alpha$ model.

\subsubsection{Time-dependent statistical solutions}

We start with the definition of Vishik-Fursikov measure for
the Navier-Stokes equations which will give rise to the statistical
solutions for the Navier-Stokes equations.

\medskip
\begin{definition}\label{vfmdef}
Let $I\subset \mathbb R$ be an interval. We say that a Borel probability measure
$\rho$ in $\mathcal C_{loc}(I,H_w)$ is a \textbf{Vishik-Fursikov measure} over
$I$ if $\rho$ satisfies the following
\begin{itemize}
 \item [(i)] $\rho$ is carried by $\mathcal U^\sharp_I$;
 \item [(ii)] $t \mapsto \int_{\mathcal U^\sharp_I}|\bu(t)|^2d\rho(\bu) \in
L^\infty_{loc}(I)$;
 \item [(iii)] if $I$ is closed and bounded on the left, with left end point
$t_0$, then for all $\psi\in \Psi$ we have that
\[\lim_{t\rightarrow t_0^+} \int_{\mathcal
U^\sharp_I}\psi(|\bu(t)|^2)d\rho(\bu)=\int_{\mathcal
U^\sharp_I}\psi(|\bu(t_0)|^2)d\rho(\bu),\]
\end{itemize}
where $\Psi:=\{\psi\in
\mathcal C^1([0,\infty)): \psi\geq 0, \psi'\geq 0 \mbox{ and }
\sup_{t\geq 0}\psi'(t) < \infty\}$.
\end{definition}

Observe that in this definition we only require the measure $\rho$ to be carried
by $\mathcal U^\sharp_I$. But we really want to have $\rho$ carried
by $\mathcal U_I$, which is the trajectory space of Leray-Hopf weak solutions.
And this is what in fact happens. More precisely, in \cite[Theorem 4.1]{FRT},
it was proved that for an arbitrary interval $I\subset \mathbb R$, any
Vishik-Fursikov measure over $I$ is carried by $\mathcal U_I$.

Next we present the definition of a Foias-Prodi statistical solution, which is a
family of measures on the phase space, satisfying a Liouville-type equation and
some regularity properties. Let us denote by $\mathbb F$ the operator defined
on $V$ as $\mathbb F(\bu)=\f-\nu A\bu-B(\bu,\bu)$, with values in $V'$. A
function $\Phi:H\rightarrow
\mathbb R$ is called a \textbf{cylindric test function} if 
\[\Phi(\bu) = \varphi((\bu, \bv_1), \ldots, (\bu, \bv_k)),\]
where $k \in \mathbb N$, $\varphi$ is a continuously differentiable real-valued
function on $\mathbb R^k$ with compact support, and $\bv_1,\ldots, \bv_k$ belong
to $V$.
For such $\Phi$, we denote by $\Phi'$ its Fr\'echet derivative in $H$, which has
the form
\[\Phi'(\bu) =\sum_{j=1}^k\partial_j\varphi ((\bu, \bv_1 ),\ldots,(\bu,
\bv_k))\bv_j,\]
where $\partial_j\varphi$ is the derivative of $\varphi$ with respect to its
$j$-th coordinate.

\medskip
\begin{definition}\label{foias-prodi}A family $\{\mu_t\}_{t\geq 0}$ of
Borel probabilities on $H$ is a \textbf{statistical solution} of
the 3D Navier-Stokes equations if it satisfies:

\begin{itemize}
 \item [(i)]the Liouville type equation 
\[ \frac{d}{d t} \int_H \Phi(\bu)\;d\mu_t(\bu) = \int_H
\langle\bF(\bu),\Phi'(\bu)\rangle\;d\mu_t(\bu),
\]
in the distributional sense in $t\geq 0$, for all cylindric test functions
$\Phi$;
 \item[(ii)] the function 
\[ t\mapsto \int_H \phi(\bu) \;d\mu_t(\bu)
\]
is measurable in $t\geq 0$ for all continuous functional $\phi:H\rightarrow
\mathbb R$;
 \item[(iii)] the function 
\[  t\mapsto \int_H |\bu|_H^2 \;d\mu_t(\bu)
\]
belongs to $L_{loc}^\infty(0,\infty)$;
\item[(iv)]the function
\[ t\mapsto \int_H \|\bu\|_V^2 \;d\mu_t(\bu)
\]
belongs to $L_{loc}^1(0,\infty)$;
\item [(v)] the mean strengthened energy inequality holds, i.e.,
\[
  \frac{1}{2} \frac{d}{d t} \int_H \psi\left(|\bu|_H^2\right)
\;d\mu_t(\bu)
     + \nu \int_H \psi'\left(|\bu|_H^2\right)
       \|\bu\|_V^2\;d\mu_t(\bu) 
     \leq \\ \int_H
\psi'\left(|\bu|_H^2\right)\langle\f,\bu\rangle\;d\mu_t(\bu)
\]
in the distributional sense in $t\geq 0$, for all $\psi \in \Psi$;
\item [(vi)] and  the function
\[ t\mapsto \int_H \psi(|\bu|_H^2) \;d\mu_t(\bu)
\]
is continuous at $t=0$, for all $\psi\in \Psi$.
\end{itemize}
\end{definition}

In \cite{FRT}, it was proved that for any Borel probability measure $\mu_0$
on $H$ such that $\int_H|u|^2d\mu_0(u)< \infty$, there exists a
Vishik-Fursikov measure $\rho$ over $I=[t_0,\infty)$ such that
$\Pi_{t_0}\rho=\mu_0$. Furthermore, if $\rho$ is a Vishik-Fursikov measure then
$\mu_t:=\Pi_t\rho$, $t\in I$, is a statistical solution in the sense of
Definition
\ref{foias-prodi}. This yields a particular
type of statistical solution:

\medskip
\begin{definition} Let $I\subset \mathbb R$ be an arbitrary interval. A
\textbf{Vishik-Fursikov statistical solution} of the Navier-Stokes equations
over $I$ is a statistical solution $\{\rho_t\}_{t\in I}$ such
that $\rho_t=\Pi_t\rho$, for all $t\in I$, for some Vishik-Fursikov
measure $\rho$ over the interval $I$.
\end{definition}

Inspired by the definition of a Vishik-Fursikov measure, we define  the
$\alpha$-Vishik-Fursikov
measure which will give rise to the statistical solutions for the
Navier-Stokes-$\alpha$ model.

\medskip
\begin{definition}\label{vfalphamdef}
Let $I\subset \mathbb R$ be an interval and $\alpha>0$. We say that a Borel
probability measure
$\rho_\alpha$ in $\mathcal C_{loc}(I,H)$ is an \textbf{$\alpha$-Vishik-Fursikov
measure} over $I$ if
$\rho_\alpha$ satisfies the following
\begin{itemize}
 \item [(i)] $\rho_{\alpha}$ is carried by $\mathcal U^{\alpha}_I$;
 \item [(ii)] $t \mapsto \int_{\mathcal
U^{\alpha}_I}|\bw(t)|^2d\rho_{\alpha}(\bw)
\in L^\infty_{loc}(I)$.
\end{itemize}
\end{definition}

Since we chose to work with solutions of the Navier-Stokes-$\alpha$
model in the sense of Definition \ref{def2}, and thanks to condition (iii) of that definition, it is natural to define $\alpha$-Vishik-Fursikov measures 
on $\mC_{loc}(I,H)$. Since $\mU_I^\alpha$ is contained in $\mC_{loc}(I,H)$ as a
set, the condition (ii) of the Definition \ref{vfalphamdef} above makes sense. 

As another remark, since $\mC_{loc}(I,H)$ is continuously included in $\mC_{loc}(I,H_w)$ and the space $\mU^\alpha_I(R)$ is closed in $\mC_{loc}(I,H_w)$ (see Section \ref{sec1.5}), there is no need to consider a space analogous to $\mU^\sharp_I$, as in Definition \ref{vfmdef}.

\medskip
\begin{remark}\label{alphatight}
Observe that $\mathcal C_{loc}(I,H)$ is a Polish space (since it is Fr\'echet
and
separable) so that every Borel probability measure $\rho$ is tight.
Moreover, since the Borel sets of $H$ are the same as the Borel sets of $H_w$
(see \cite{fmrt2001a}) then the Borel sets of $\mathcal C_{loc}(I, H)$ are the
same as the Borel sets of $\mathcal C_{loc}(I, H_w)$. Therefore, every 
Borel probability measure in $\mathcal C_{loc}(I,H_w)$ is tight.
\end{remark}

We also define an \textbf{$\alpha$-Vishik-Fursikov statistical solution} of the
Navier-Stokes-$\alpha$ model over an arbitrary interval $I$ as a family
$\{\rho^\alpha_t\}_{t\in I}$  of Borel probability measures on $H$ such
that $\rho^\alpha_t=\Pi_t\rho_\alpha$, for all $t\in I$, for some 
$\alpha$-Vishik-Fursikov measure $\rho_\alpha$ over the interval $I$.

The existence of $\alpha$-Vishik-Fursikov measures is easy to obtain.
For instance, any Dirac measure in the trajectory space $\mathcal U_I^\alpha$ is
an $\alpha$-Vishik-Fursikov measure over $I$. Furthermore, given an initial
Borel probability  measure $\mu_0$ in $H$ with finite energy we can construct an
$\alpha$-Vishik-Fursikov measure $\rho_\alpha$ over $[0,\infty)$ such that
$\Pi_{0}\rho_\alpha=\mu_0$. Indeed, since the Navier-Stokes-$\alpha$ model is
well-posed, the solution semigroup $\{S_{\alpha}(t)\}_{t\geq 0}$ is
well-defined. Moreover, the operator
$S_{\alpha}(\cdot): H\rightarrow \mathcal C_{loc}(I,H)$, defined as
$S_\alpha(\cdot)\bw_0=\bw(\cdot)$ for $\bw_0\in H$, where
$S_{\alpha}(t)\bw_0=\bw(t)$ for all $t\in I$,  is continuous.
Therefore, given an initial Borel  probability  measure $\mu_0$ on $H$, we may
define the Borel  probability  measure $\rho_{\alpha}$ as
\[\rho_{\alpha}(E)=\mu_0(S_\alpha(\cdot)^{-1}E), \quad \mbox{ for all } E\in
\mathcal C_{loc}(I,H) \mbox{ Borel,  where } I=[0,\infty).\]
By construction,  it is clear that $\rho_\alpha$ is carried by the set $\mathcal
U^{\alpha}_I$. Moreover, since for each $\bw\in \mathcal U^\alpha_I$ it holds
that $|\bw(t)|\leq |\bw(0)| + 1/(\lambda_1^2\nu_1^2)\|\f\|^2_{L^\infty(I,H)}$,
for all $t\geq 0$, and
$\rho_{\alpha}=S_\alpha(\cdot)\mu_0$, then, using the Change of Variables
Theorem (see Section \ref{compactness}), we obtain that
\begin{equation*}
\begin{split}
\int_{\mathcal
U^{\alpha}_I}|\bw(t)|^2d\rho_\alpha(\bw)&\leq \int_{\mathcal
U^{\alpha}_I}|\bw(0)|^2d\rho_\alpha(\bw) +
\frac{1}{\lambda_1^2\nu_1^2}\|\f\|^2_{L^\infty(I,H)}
\\&=\int_{H} |(S_\alpha(\cdot)\bu)(0)|^2d\mu_0(\bu)+
\frac{1}{\lambda_1^2\nu_1^2}\|\f\|^2_{L^\infty(I,H)} 
\\&=\int_{H}
|\bu|^2d\mu_0(\bu)+
\frac{1}{\lambda_1^2\nu_1^2}\|\f\|^2_{L^\infty(I,H)}.
\end{split}
\end{equation*}
Consequently, since $\mu_0$ has finite energy we obtain that $t \mapsto
\int_{\mathcal U^{\alpha}_I}|\bw(t)|^2d\rho_\alpha(\bw)
\in L^\infty_{loc}([0,\infty))$. Therefore, $\rho_\alpha$ is an
$\alpha$-Vishik-Fursikov measure, and it is straightforward that
$\Pi_{0}\rho_\alpha=\mu_0$.

Nevertheless, since the well-posedness for the 3D
Navier-Stokes equations is an open problem, the abstract definition of
Vishik-Fursikov measure is essential in the context of this article.

\subsubsection{Stationary statistical solutions}

As already mentioned in the Introduction, stationary statistical solutions are
valuable in the study of turbulence in statistical equilibrium in time,
yielding, in particular, rigorous proofs of important statistical estimates. 
The concept of stationary statistical
solution represents a generalization of invariant measures for the semigroup
generated by an equation. For instance, in 2D, since the Navier-Stokes equations
has a well-defined semigroup, stationary statistical solutions are, under
certain hypothesis, precisely the invariant measures for the semigroup; see
\cite{fmrt2001a} for more details.

Suppose $I$ is an interval unbounded on the right, hence having one of the following forms:
$I=[t_0,\infty)$, $I=(t_0,\infty)$ or $I = \mathbb R$. We introduce the
time-shift operator $\sigma_{\tau}$ defined for any $\tau >0$ by
\begin{eqnarray*}
\sigma_{\tau}: \mathcal C_{loc}(I, H_w)&\rightarrow & C_{loc}(I, H_w) \\
\bu&\mapsto & (\sigma_{\tau}\bu)(t)=\bu(t+\tau), \; \forall t\in I.
\end{eqnarray*}
An \textbf{invariant
Vishik-Fursikov measure} over $I$ is a Vishik-Fursikov measure $\rho$  which is
invariant with respect to the translation semigroup $\{\sigma_\tau\}_{\tau\geq
0}$, in the sense that $\sigma_{\tau}\rho=\rho$ for all $\tau \geq 0$,
i.e., $\rho(E)= \rho(\sigma_\tau^{-1} E)$, for all Borel set $E$ in $\mathcal
C_{loc}(I, H_w)$.

The family of projections $\{\Pi_t\rho\}_{t\in I}$ of an invariant
Vishik-Fursikov measure $\rho$ has the property that 
any statistical information
\[ \int_H \phi(\bu) \rd(\Pi_t\rho)(\bu) = \int_H \phi(\bu(t))
\rd\rho(\bu) 
\]
is independent of the time variable $t\in I$, for any $\phi\in\mathcal
C_b(H_w)$. In fact, the measure $\Pi_t\rho$ itself is independent of $t$ and
is a statistical solution in the sense of Foias-Prodi which is time-independent
(what is called
a stationary statistical solution in the sense of Foias-Prodi). This yields
a particular type of stationary statistical solution as stated precisely below.

\medskip
\begin{definition}A \textbf{stationary Vishik-Fursikov statistical solution} on
$H$ is a Borel probability measure $\rho_0$ on $H$ which is a projection
$\rho_0=\Pi_t\rho$, at an arbitrary time $t\in I$, of an invariant
Vishik-Fursikov measure $\rho$ over an interval $I$ unbounded on the right.
\end{definition}

\section{Convergence of statistical solutions of the $\alpha$-model as
     $\alpha$ vanishes}\label{sec2}

In this section we present the main results of this paper. We prove that, under
certain conditions, statistical solutions of the Navier-Stokes-$\alpha$ model
converge to statistical solutions of the 3D Navier-Stokes equations as
$\alpha\rightarrow
0$. We first prove this result for  time-dependent statistical
solutions and then we address the particular case of stationary statistical
solutions.

As defined in Section \ref{compactness}, we mean by $\rho_\alpha
\stackrel{*}{\rightharpoonup} \rho$ in
$\mathcal P(X)$  that
\[\lim_{\alpha\rightarrow 0}
\int_X\varphi(x)\rd\rho_\alpha(x)=\int_X\varphi(x)\rd\rho(x), \mbox{ for all
}\varphi\in \mathcal C_b(X).\]
Observe that, by Theorem \ref{part-topsoe}, in the case when $X$ is a completely
regular Hausdorff space, which is the case for us, this is exactly the weak convergence $\wconv$ discussed in Section \ref{compactness}.

\subsection{Time-dependent statistical solutions}

In this section we state some results concerning the
convergence of $\alpha$-Vishik-Fursikov measures and statistical solutions as
$\alpha$ vanishes. We will
see that it suffices to impose some condition of uniform boundedness on
the mean kinetic energy over the $\alpha$-Vishik-Fursikov measures (or,
equivalently, on the $\alpha$-Vishik-Fursikov statistical solutions). This will
assure the tightness of the $\alpha$-Vishik-Fursikov measures and yield a
convergent subsequence. The convergence of these measures will imply the
convergence of the associated statistical solutions in a suitable sense,
described as follows.

Let $I\subset \mathbb R$ be an arbitrary interval.  We say
that a family $\{\{\mu^{\alpha}_t\}_{t\in I}\}_{\alpha>0}$ of
$\alpha$-Vishik-Fursikov
statistical solutions over $I$ converges, as $\alpha\rightarrow 0$, to a
Vishik-Fursikov statistical solution over $I$ if there exists a Vishik-Fursikov
statistical solution $\{\mu_t\}_{t\in I}$ over $I$ such that
\begin{equation}\label{stat-conv}
\lim_{\alpha\rightarrow 0}\int_{H}\phi(\bu)d\mu^{\alpha}_t(\bu)=\int_{H}
\phi(\bu)d\mu_t(\bu), \quad \forall t\in I \mbox{ and }\forall \phi \in
\mathcal C_b(H_w). 
\end{equation}

Recall that a Vishik-Fursikov statistical solution ($\alpha$-Vishik-Fursikov
statistical solution) over an interval $I$ is a family of Borel probability
measures $\{\mu_t\}_{t\in I}$ given by $\mu_t=\Pi_t\rho$, for all $t\in I$,
where $\rho$ is some Vishik-Fursikov measure ($\alpha$-Vishik-Fursikov
measure) over $I$. 

Thus, if we have a family of $\alpha$-Vishik-Fursikov measures
$\{\rho_{\alpha}\}_{\alpha>0}$ that converges to a Vishik-Fursikov measure
$\rho$
in $\mathcal P(\mathcal Y_I)$, i.e.,
\begin{equation}\label{mesu-conv}
\lim_{\alpha\rightarrow 0}\int_{\mathcal
Y_I}\varphi(\bu)d\rho_{\alpha}(\bu)=\int_{\mathcal Y_I}
\varphi(\bu)d\rho(\bu), \quad \forall \varphi \in
\mathcal C_b(\mathcal Y_I),
\end{equation}
then, as a simple consequence of the
Change of Variables Theorem (see Section \ref{compactness}), we obtain the
convergence of the $\alpha$-Vishik-Fursikov
statistical solutions  $\{\Pi_t\rho_{\alpha}\}_{t\in I}$ to the
Vishik-Fursikov statistical solution $\{\Pi_t\rho\}_{t\in I}$. 

The measures $\rho$ and $\rho_\alpha$ are defined on $\mC_{loc}(I,H_w)$ and
$\mC_{loc}(I,H)$, respectively. Since $\mC_{loc}(I,H)$ is included in
$\mC_{loc}(I,H_w)$ we may consider them  as measures on $\mC_{loc}(I,H_w)$. In
fact, since they are actually carried by the space of solutions $\mathcal U_I$
and $\mathcal U_I^\alpha$, respectively, and these spaces are included in
$\mY_I$, we may consider them restricted to $\mY_I$. With this in mind and for
the sake of simplicity we consider, in what follows, the measures $\rho$ and
$\rho_\alpha$ as measures on $\mY_I$. This is the same reason why we used
$\mY_I$ in the convergence \eqref{mesu-conv}.

We state first a lemma which is  essentially Theorem \ref{part-topsoe}
translated
into the framework of interest in this section. 

\medskip
\begin{lemma}\label{mtheo4}
Let $I$ be an arbitrary interval in $\mathbb R$ and 
$\{\rho_{\alpha}\}_{\alpha>0}$ be a family in $\mathcal P(\mathcal Y_I,t)$.
Suppose
that for every $\varepsilon >0$ there exists $R>0$ such that
$\rho_{\alpha}(\mathcal Y_I\setminus \mathcal Y_I(R))<\varepsilon$, for all
$\alpha>0$. Then, there exists a sequence $\{\rho_{\alpha_{n}}\}_n\subset
\{\rho_{\alpha}\}_{\alpha>0}$, with $\alpha_n\rightarrow 0$ as $n\rightarrow \infty$, such that $\rho_{\alpha_{n}}
\stackrel{*}{\rightharpoonup} \rho$ in $\mathcal P(\mathcal Y_I)$ as
$n\rightarrow \infty$.
\end{lemma}

\begin{proof}
First, observe that since $\{\alpha\in \mathbb R: \alpha>0\}$ is a
totally ordered set we can subtract a sequence 
$\{\rho_{\alpha_{n}}\}_n$ from $\{\rho_{\alpha}\}_{\alpha}$, with
$\alpha_n\rightarrow 0$ as
$n\rightarrow \infty$. Moreover, observe that $\mathcal Y_I$ is a completely
regular Hausdorff space and  
$\mathcal Y_I(R)$ is compact, for any $R>0$. Therefore, it is clear
that the sequence $\{\rho_{\alpha_n}\}_{n\in \mathbb N}$ fulfills the hypothesis
of Theorem \ref{part-topsoe}. Therefore, there exist
$\rho \in \mathcal P(\mathcal Y_I,t)$ and a subsequence, which we still denote
by $\{\rho_{\alpha_n}\}_{n}\subset \{\rho_{\alpha}\}_{\alpha>0}$, such
that $\rho_{\alpha_{n}}\stackrel{*}{\rightharpoonup} \rho$ in $\mathcal
P(\mathcal Y_I)$. 
\end{proof}

We would like to highlight that the use of a sequence in Lemma \ref{mtheo4} is only for the sake of simplicity. Notice that Theorem \ref{part-topsoe} can be applied to the net $\{\rho_\alpha\}_\alpha$, as in Lemma \ref{mtheo4}, and it yields the existence of a convergent subnet of $\{\rho_\alpha\}_\alpha$. Then, all the subsequent results are also valid if we work with a convergent subnet instead of a sequence. 

Observe that in order to apply the last lemma to a family of $\alpha$-Vishik
Fursikov measures we need the sequence to be uniformly tight. A natural 
condition that yields this uniform tightness is obtained by
imposing a uniform boundedness condition on the mean kinetic energy of this
family. This uniform boundedness of the mean kinetic energy will also be
important to yield that the limit measure has finite mean kinetic energy and is
a Vishik-Fursikov measure. This uniform boundedness can be imposed in
different ways. We start with the following:

\medskip
\begin{proposition}\label{prop5}
Let $I$ be any interval in $\mathbb R$. Let $\{\rho_\alpha\}_{\alpha>0}$ be a
family of Borel probability measures on
$\mathcal Y_I$ such that, for each $\alpha>0$, $\rho_{\alpha}$ is
an $\alpha$-Vishik-Fursikov measure over $I$ and $\sup_{t\in
I}\int_{\mathcal Y_I}|\bw(t)|^2d\rho_{\alpha}(\bw)\leq C$, for all
$\alpha>0$, for some constant $C\geq 0$.
Then, there exists a sequence $\{\rho_{\alpha_n}\}_n\subset
\{\rho_{\alpha}\}_{\alpha>0}$, with $\alpha_n\rightarrow 0$, converging to
a Borel probability measure $\rho$ in $\mathcal Y_I$. Moreover, $\rho$ is
carried by $\mathcal U_I^\sharp$ and 
$\Pi_{\mathring{I}}\rho$ is a Vishik-Fursikov measure over $\mathring{I}$.
\end{proposition}

\begin{proof}
In order to prove the convergence we will use Lemma \ref{mtheo4}. First, recall
that every $\alpha$-Vishik-Fursikov measure is tight (Remark
\ref{alphatight}). Next, we check the uniform tightness condition. Observe that
it is enough to prove that for all $\varepsilon >0$ there exists $R\geq R_0$
such
that $\rho_{\alpha}(\mathcal U_I^{\alpha}\setminus \mathcal
U_I^{\alpha}(R))<\varepsilon$  for all $\alpha>0$. Indeed, since
$\rho_{\alpha}$ is carried by $\mathcal
U_I^{\alpha}$ we have that
\begin{equation*}
\begin{split}
\rho_{\alpha}(\mathcal Y_I\setminus \mathcal
Y_I(R))&=\rho_{\alpha}(\mathcal
Y_I\cap (\mathcal Y_I(R))^c)=\rho_{\alpha}(\mathcal
U_I^{\alpha}\cap (\mathcal Y_I(R))^c) 
\\ &\leq\rho_{\alpha}(\mathcal
U_I^{\alpha}\cap (\mathcal 
U_I^{\alpha}(R))^c)=\rho_{\alpha}(\mathcal
U_I^{\alpha}\setminus \mathcal 
U_I^{\alpha}(R)),
\end{split}
\end{equation*}
where the inequality holds since $\mathcal U_I^{\alpha}(R)\subset \mathcal
Y_I(R)$, which implies that $\mathcal Y_I(R)^c \subset (\mathcal
U_I^{\alpha}(R))^c$.

Let $R\geq R_0$. If $\mathcal U_I^{\alpha}\setminus \mathcal
U_I^{\alpha}(R)$
is not empty, then for all $\bw \in \mathcal
U_I^{\alpha}\setminus \mathcal U_I^{\alpha}(R)$, there exists $t_{\bw}\in I$
such that
$|\bw(t_{\bw})|^2>R$. Thus, $\sup_{t\in I}|\bw(t)|^2\geq R$ and 
\begin{equation}\label{est-rho}
R\rho_{\alpha}(\mathcal U^{\alpha}_I\setminus \mathcal
U^{\alpha}_I(R))\leq \int_{\mathcal U_I^{\alpha}\setminus \mathcal 
U_I^{\alpha}(R)}\sup_{t\in I}|\bw(t)|^2d\rho_{\alpha}(\bw).
\end{equation}

Since for all $\bw \in \mathcal U_I^{\alpha}$ we have that 
\[|\bw(t)|^2\leq |\bw(t')|^2e^{-\lambda_1\nu
(t-t')}+\frac{1}{\lambda_1^2\nu^2}\|\f\|^2_{L^{\infty}(t',t;
H)}\left(1-e^{-\lambda_1\nu(t-t')}\right),\]
for all $t,t'\in I$ with $t>t'$, then
\[\int_{\mathcal U_I^{\alpha}}\sup_{t>t'}|\bw(t)|^2d\rho_{\alpha}\leq
\int_{\mathcal U_I^{\alpha}} |\bw(t')|^2d\rho_{\alpha}
+\frac{1}{\lambda_1^2\nu^2}\|\f\|^2_{L^\infty(I,H)}.\]
Therefore, by the last inequality and the hypothesis $\sup_{t\in
I}\int_{\mathcal
Y_I}|\bw(t)|^2d\rho_{\alpha}(\bw)\leq C$, for all $\alpha>0$, it follows
that
\[\int_{\mathcal U_I^{\alpha}}\sup_{t>t'}|\bw(t)|^2d\rho_{\alpha}\leq
C_1,\]
for all $t'\in I$ and for all $\alpha>0$, where
$C_1=C+1/(\lambda_1^2\nu^2)\|\f\|^2_{L^\infty(I,H)}$.

Also observe that, for all $\bw \in \mathcal
U_I^{\alpha}$,  $\bw \in \mathcal C_{loc}(I,H)$ so that the function defined
by $f(t')=\sup_{t>t'}|\bw(t)|^2$
for $t'\in \bar{I}$, is continuous. Thus, for any $t_0\in \bar{I}$
and any sequence $\{t'_k\}_k\subset I$ such that $t'_k\rightarrow t_0^+$, using
the Monotone Convergence Theorem, we find that
\[\int_{\mathcal U_I^{\alpha}}\sup_{t>t_0}|\bw(t)|^2d\rho_{\alpha}\leq
\int_{\mathcal U_I^{\alpha}} |\bw(t')|^2d\rho_{\alpha}
\leq C_1, \quad \forall t_0\in \bar{I} \mbox{ and }\forall \alpha>0.\]
In particular, we obtain 
\[\int_{\mathcal U_I^{\alpha}}\sup_{t\in I}|\bw(t)|^2d\rho_{\alpha}\leq
C_1, \quad \forall\alpha>0.\]

We use the last estimate in \eqref{est-rho} to obtain, in the case that
$\mathcal U_I^{\alpha}\setminus \mathcal U_I^{\alpha}(R)$ is not empty, that
\[\rho_{\alpha}(\mathcal U^{\alpha}_I\setminus \mathcal
U^{\alpha}_I(R))\leq\frac{C_1}{R}, \qquad \forall \alpha>0.\]
If $\mathcal U_I^{\alpha}\setminus \mathcal U_I^{\alpha}(R)$ is empty, then
this estimate is trivially valid.
Then, given $\varepsilon >0$, take $R=\max\{2C_1/\varepsilon, R_0\}$, so that
 
\[\rho_{\alpha}(\mathcal U^{\alpha}_I\setminus \mathcal
U^{\alpha}_I(R))<\varepsilon, \qquad \forall \alpha>0.\]
This shows that $\{\rho_{\alpha}\}_{\alpha>0}$ is uniformly tight.

Then, we can apply Lemma \ref{mtheo4} and obtain the existence of a measure
$\rho\in\mathcal P(\mathcal Y_I,t)$ and a sequence
$\{\rho_{\alpha_n}\}_n\subset\{\rho_{\alpha}\}_{\alpha>0}$, with
$\alpha_n\rightarrow 0$ as $n\rightarrow \infty$, such that
$\rho_{\alpha_n}\stackrel{*}{\rightharpoonup} \rho$ in $\mathcal P(\mathcal
Y_I)$ as $n\rightarrow \infty$. 

Next, we prove that $\rho$ is carried by $\mU_I^\sharp$. In order to do so,
we define, for each $\varepsilon>0$ and $R\geq R_0$, the set  
\[\mY_{\varepsilon}(R)=\{u\in \mathcal Y_I(R): \dist_{\mathcal
Y_I(R)}(u,\mathcal
U_I^\sharp(R))<\varepsilon\}.\]
Observe that $\mathcal Y_I(R)\setminus \mY_{\varepsilon}(R)$ and $\overline{
\mY_{\varepsilon/2}(R)}$ are disjoint closed sets in $\mathcal Y_I(R)$. Since
$\mathcal Y_I(R)$ is a compact Hausdorff space then, by
Urysohn's Lemma (see
e.g. \cite{AB}), for each $\varepsilon>0$, there exists a continuous function
$\varphi_\varepsilon^R: \mathcal Y_I(R)\rightarrow [0,1]$, such that
$\varphi_\varepsilon^R(\bu)=1$ for all $\bu\in
\overline{\mY_{\varepsilon/2}(R)}$
and $\varphi_\varepsilon^R(\bu)=0$ for all $\bu \in\mathcal Y_I(R)\setminus
\mY_{\varepsilon}(R)$. Now, we define an extension of $\varphi_\epsilon^R$ to
$\mathcal Y_I$, $ \tilde \varphi_\varepsilon^R: \mathcal
Y_I\rightarrow [0,1]$ as $\tilde
\varphi_\varepsilon^R(\bu)=\varphi_\varepsilon^R(\bu)$, for all $\bu \in
\mathcal Y_I(R)$ and $\tilde \varphi_\varepsilon^R(\bu)=0$, for all $\bu \in
\mathcal Y_I\setminus  \mathcal Y_I(R)$. Since $\mY_I(R)$ is closed, it is easy to see that $\tilde
\varphi_\varepsilon^R$ is upper semicontinuous. Then, using that
$\rho_{\alpha_n}\stackrel{*}{\rightharpoonup} \rho$ in $\mathcal P(\mathcal
Y_I)$  and 
Lemma \ref{portmanteau}, we obtain that
\begin{equation}\label{est-int}
\int_{\mathcal Y_I}\tilde
\varphi_\varepsilon^R(\bu)d\rho(\bu)\geq
\limsup_{n\rightarrow \infty}\int_{\mathcal Y_I}\tilde
\varphi_\varepsilon^R(\bu)d\rho_{\alpha_n}(\bu)=\limsup_{n\rightarrow
\infty}\int_{\mathcal
U_I^{\alpha_n}(R)}\varphi_\varepsilon^R(\bu)d\rho_{\alpha_n}(\bu).
\end{equation}
Since we have, by Lemma \ref{dist}, that $\lim_{n\rightarrow
\infty}\dist_{\mathcal Y_I(R)}(\mathcal
U_I^{\alpha_n}(R),\mathcal U_I^\sharp(R))=0$ then, given $\varepsilon >0$, there
exists $N_0\in \mathbb N$ such that, for each $m\geq N_0$,
\[\dist_{\mathcal Y_I(R)}(\mathcal U_I^{\alpha_m}(R),\mathcal
U_I^\sharp(R))<\varepsilon/2,\]
which implies that for all $\bu \in \mathcal U_I^{\alpha_m}(R)$, $m\geq N_0$,
$\dist_{\mathcal Y_I(R)}(\bu,\mathcal
U_I^\sharp(R))<\varepsilon/2$. In other words, given $\varepsilon>0$ there
exists $N_0\in \mathbb N$ such that $\mathcal U_I^{\alpha_m}(R)\subset
\mY_{\varepsilon/2}(R)$, for all $m\geq N_0$ and therefore
$\varphi_\varepsilon^R(\bu)=1$ for all $\bu \in \mathcal U_I^{\alpha_m}(R)$, for
all $m\geq N_0$. Thus, for all $m\geq N_0$,
\[\int_{\mathcal
U_I^{\alpha_m}(R)}\varphi_\varepsilon^R(\bu)d\rho^{\alpha_m}(\bu)=\rho_{
\alpha_m}(\mathcal U_I^{\alpha_m}(R))\geq 1-\frac{C_1}{R},\]
hence, using the last estimate in \eqref{est-int}, we conclude that
\[\int_{\mathcal Y_I}\tilde \varphi_\varepsilon^R(\bu)d\rho(\bu)\geq
1-\frac{C_1}{R}.\]
Therefore, 
\[\rho(\mY_\varepsilon(R))\geq \int_{\mathcal Y_I}\tilde
\varphi_\varepsilon^R(\bu)d\rho(\bu)\geq
1-\frac{C_1}{R}.\]
Since $\mathcal U_I^\sharp(R) = \cap_{j=1}^\infty \mY_{\varepsilon_j}(R)$ for
any
sequence of positive numbers $\varepsilon_j \rightarrow 0$, we
obtain that

\[\rho(\mathcal U_I^\sharp(R))\geq 1-C_1/R, \quad \forall R\geq R_0.\] 
Now, in order to prove that $\rho(\mathcal U^\sharp_I)=1$, we first suppose
that $I$ is closed and bounded on the left. In this case we can
write $\mathcal U_I^\sharp= \cup_{i=1}^\infty \mathcal U_I^\sharp(R_i)$, for
any sequence $\{R_i\}_i$ with
$R_0\leq R_i\leq R_{i+1}$, for all $i\in \mathbb N$, and $R_i\rightarrow
\infty$.
Observe that $\mathcal U_I^\sharp(R_i)\subset \mathcal
U_I^\sharp(R_{i+1})$, for all
$i\in \mathbb N$. Therefore, 
\[\rho(\mathcal U^\sharp_I)=\rho\left(\bigcup_{i=1}^\infty \mathcal
U_I^\sharp(R_i)\right)=\lim_{i\rightarrow
\infty}\rho(\mathcal U_{I}^\sharp(R_i))\geq\lim_{i\rightarrow
\infty}\left(1-\frac{C_1}{R_i}\right)=1.\]

Otherwise, we can write $\mathcal
U^\sharp_I=\bigcap_{n=1}^\infty\bigcup_{i=1}^\infty
\Pi^{-1}_{J_n}\mathcal U^\sharp_{J_n}(R_i)$, for any sequence $\{R_i\}_i$ with
$R_0\leq R_i\leq R_{i+1}$, for all $i\in \mathbb N$, $R_i\rightarrow \infty$
and for any sequence  $\{J_n\}_n$ of compact subintervals of $I$ such that
$J_{n}\subset J_{n+1}$, for all $n\in \mathbb N$, and $\cup_n J_n=I$. Observe
that, for all compact
interval $J\subset I$ it is true that 
$\mathcal U_I^\sharp(R)\subset \Pi_J^{-1}\mathcal U^\sharp_{J}(R)$ so that
$\rho(\Pi^{-1}_{J}\mathcal U^\sharp_{J}(R))\geq 1-C_1/R$. Moreover, it is
easy to see the monotonicity properties
 \[ \bigcup_{i=1}^\infty
\Pi^{-1}_{J_{n}}\mathcal U^\sharp_{J_{n}}(R_i) \supset \bigcup_{i=1}^\infty
\Pi^{-1}_{J_{n+1}}\mathcal U^\sharp_{J_{n+1}}(R_i),\]
for all $n\in \mathbb N$, and 
\[\Pi^{-1}_{J_n}\mathcal U^\sharp_{J_{n}}(R_i)\subset \Pi^{-1}_{J_n}\mathcal
U^\sharp_{J_{n}}(R_{i+1}),\] 
for all $n, i \in \mathbb N$. Therefore,

\begin{equation*}
\begin{split}
\rho(\mathcal U^\sharp_I)&=\rho\left(\bigcap_{n=1}^\infty\bigcup_{i=1}^\infty
\Pi^{-1}_{J_n}\mathcal U^\sharp_{J_n}(R_i)\right)=\lim_{n\rightarrow
\infty}\rho\left(\bigcup_{i=1}^\infty\Pi^{-1}_{J_n}\mathcal
U_{J_n}^\sharp(R_i)\right)\\&=
\lim_{n\rightarrow \infty}\lim_{i\rightarrow
\infty}\rho(\Pi^{-1}_{J_n}\mathcal U_{J_n}^\sharp(R_i))
\geq \lim_{n\rightarrow \infty}\lim_{i\rightarrow
\infty}\left(1-\frac{C_1}{R_i}\right)=1.
\end{split}
 \end{equation*}
Thus, $\rho(\mathcal U^\sharp_I)=1$.

It remains to prove that $t\mapsto \int_{\mathcal Y_I}|\bu(t)|^2d\rho(\bu)$
belongs to $L_{loc}^{\infty}(I)$. In that direction,  we define an
increasing sequence of cut-off functions $\{\phi_M\}_M$, that is,
for all $M>0$ we have that $\phi_M\in \mC^\infty([0,\infty))$,
$0\leq \phi_M\leq 1$, $\phi_M(r)=1$ for $0\leq r\leq M$, $\phi(r)=0$
for $r\geq 2M$ and $\phi_M\leq \phi_{M+1}$. Now, take any $t'\in I$ and
observe that the function
$f_{M,k}$, defined by $f_{M,k}(\bu)=\phi_M(|P_k\bu(t')|^2)|P_k\bu(t')|^2$,
where $P_k$ is the Galerkin projector (see Section \ref{sec1.1}),
belongs to $\mC_b(\mY_I)$ and $f_{M,k}(\bu) \leq |\bu(t')|^2$, for all $\bu\in
\mathcal Y_I$. Then, from the convergence of $\rho_{\alpha_n}$ to
$\rho$ together with the hypothesis $\sup_{t\in I}\int_{\mathcal
Y_I}|\bw(t)|^2d\rho_{\alpha_n}(\bw)\leq C$, for all $n\in \mathbb N$, we
obtain
\[\int_{\mY_I}f_{M,k}(\bu)d\rho(\bu)=\lim_{n\rightarrow
\infty}\int_{\mY_I}f_{M,k}(\bu)d\rho_{\alpha_n}(\bu)\leq \limsup_{n\rightarrow
\infty}\int_{\mY_I}|\bu(t')|^2d\rho_{\alpha_n}(\bu)\leq C.\]

We can pass to the limit in the last inequality as $M\rightarrow \infty$
and, using the Monotone Convergence Theorem, we obtain that
\[\int_{\mY_I}|P_k\bu(t')|^2d\rho(\bu)=\lim_{M\rightarrow
\infty}\int_{\mY_I} f_{M,k}(\bu)d\rho(\bu)\leq C.\]
Again, using the Monotone Convergence Theorem, we can pass to the limit as
$k\rightarrow \infty$ to find that
\[\int_{\mY_I}|\bu(t')|^2d\rho(\bu)=\lim_{k\rightarrow
\infty}\int_{\mY_I}|P_k\bu(t')|^2d\rho(\bu) \leq C.\]
Now, since $t'\in I$ is arbitrary we obtain that $t\mapsto
\int_{\mathcal Y_I}|\bu(t)|^2d\rho(\bu)$ belongs to $L^{\infty}(I)$.
\end{proof}

The previous result has a corresponding statement in terms of statistical
solutions, which we write as follows.
\medskip
\begin{proposition}\label{prop5nv}
Let $I$ be any interval in $\mathbb R$ and let
$\{\{\mu_t^{\alpha}\}_{t\in I}\}_{\alpha>0}$ be a family of 
$\alpha$-Vishik-Fursikov statistical solutions over $I$, such that $\sup_{t\in
I}\int_{H}|\bw|^2d\mu_t^{\alpha}(\bw)\leq C$, for all $\alpha>0$, for some
constant $C\geq 0$.
Then, there exists a sequence $\{\{\mu_t^{\alpha_n}\}_{t\in I}\}_n\subset
\{\{\mu_t^{\alpha}\}_{t\in I}\}_{\alpha>0}$, with $\alpha_n\rightarrow 0$,
converging to a Vishik-Fursikov statistical solution $\{\mu_t\}_{t\in \mathring
I}$. 
\end{proposition}

\begin{proof}
By definition, for each $\alpha>0$, there exists an $\alpha$-Vishik-Fursikov
measure $\rho_{\alpha}$ over $I$ such that
$\mu_t^{\alpha}=\Pi_t\rho_{\alpha}$,
for all $t\in I$. Observe that the family $\{\rho_{\alpha}\}_{\alpha>0}$
fulfills
the hypothesis of Proposition \ref{prop5}, then there exists a sequence
$\{\rho_{\alpha_{n}}\}_n\subset \{\rho_{\alpha}\}_{\alpha>0}$, with
$\alpha_n\rightarrow 0$, that converges to a Borel probability measure 
$\rho$ in $\mathcal Y_I$ such that $\Pi_{\mathring{I}}\rho$ is a
Vishik-Fursikov
measure over $\mathring{I}$. Thus, as a simple consequence of the
Change
of Variables Theorem (see Section \ref{compactness}), we obtain the convergence
of the $\alpha_n$-Vishik-Fursikov
statistical solutions  $\{\Pi_t\rho_{\alpha_n}\}_{t\in I}$ to the
Vishik-Fursikov statistical solution $\{\Pi_t\rho\}_{t\in \mathring{I}}$. In
other words, $\{\{\mu_t^{\alpha_{n}}\}_{t\in I}\}_k$ converges to the
Vishik-Fursikov statistical solution $\{\mu_t\}_{t\in \mathring{I}}$, where
$\mu_t=\Pi_t\rho$, for all $t\in \mathring{I}$.
\end{proof}

A different way of obtaining the uniform tightness condition in the family
of $\alpha$-Vishik-Fursikov measure is to assume that the interval $I$ is
bounded and closed on the left and impose uniform boundedness on the initial
mean kinetic energy over the $\alpha$-Vishik-Fursikov measures. This is done in
the next result:

\medskip
\begin{corollary}\label{cor6}Let $I$ be an interval in $\mathbb R$ which is
bounded and closed on the left, with left end point $t_0$. Let
$\{\rho_{\alpha}\}_{\alpha>0}$ be a family of $\alpha$-Vishik-Fursikov measures
over $I$ such that $\int_{\mathcal
Y_I}|\bw(t_0)|^2d\rho_{\alpha}(\bw)\leq C$, for all $\alpha>0$, for
some constant $C\geq 0$. Then, there exists a sequence
$\{\rho_{\alpha_{n}}\}_n\subset \{\rho_{\alpha}\}_{\alpha>0}$, with
$\alpha_n\rightarrow 0$, that converges to a Borel probability
measure $\rho$ in $\mathcal Y_I$ such that 
$\rho$ is carried by $\mathcal U_I^\sharp$ and 
 $\Pi_{\mathring{I}}\rho$ is a Vishik-Fursikov measure over
$\mathring I$.
\end{corollary}

\begin{proof} We only need to check that the family $\{\rho_{\alpha}\}$
fulfills the hypothesis of Proposition \ref{prop5}. In order to do so, observe
that for all $\bw \in \mathcal U_I^{\alpha}$ and for all $t\in I$ with $t\geq
t_0$,
\[|\bw(t)|^2\leq |\bw(t_0)|^2e^{-\lambda_1\nu
(t-t_0)}+\frac{1}{\lambda_1^2\nu^2}\|\f\|^2_{L^{\infty}(t_0,t;
H)}(1-e^{-\lambda_1\nu(t-t_0)}),\]
therefore
\[\int_{\mathcal U_I^{\alpha}}|\bw(t)|^2d\rho_{\alpha}(\bw)\leq
\int_{\mathcal
U_I^{\alpha}}|\bw(t_0)|^2d\rho_{\alpha}(\bw)+\frac{1}{\lambda_1^2\nu^2}
\|\f\|^2_{L^{\infty}(t_0,t;H)}.\]
Since by hypothesis $\int_{\mathcal
Y_I}|\bw(t_0)|^2d\rho_{\alpha}(\bw)\leq C$, for all $\alpha>0$, and
$\rho_{\alpha}$ is carried by $\mU^{\alpha}_I$, then
\[\sup_{\alpha>0}\sup_{t\in I}\int_{\mathcal
Y_I}|\bw(t)|^2d\rho_{\alpha}(\bu)\leq
C_1,\]
where $C_1=C+1/(\nu^2\lambda_1^2)\|f\|^2_{L^\infty(I, H)}$.
\end{proof}

As before, the  previous result has a corresponding statement in terms of
statistical solutions, which we write as follows.
\medskip
\begin{corollary}\label{cor6nv}
Let $I$ be any interval in $\mathbb R$ and let $\{\{\mu_t^{\alpha}\}_{t\in
I}\}_{\alpha>0}$ be a family of 
$\alpha$-Vishik-Fursikov statistical solutions over $I$, such that
$\int_{H}|\bw|^2d\mu_{t_0}^{\alpha}(\bw)\leq C$, for all $\alpha>0$, for
some constant $C\geq 0$. 
Then, there exists a sequence $\{\{\mu_t^{\alpha_n}\}_{t\in
I}\}_{n}\subset\{\{\mu_t^{\alpha}\}_{t\in I}\}_{\alpha>0}$, with
$\alpha_n\rightarrow 0$, that converges to
a Vishik-Fursikov statistical solution $\{\mu_t\}_{t\in \mathring I}$. 
\end{corollary}

In the last two results we obtain a
limit of $\alpha$-Vishik-Fursikov measures over an interval $I$ which is a
Vishik-Fursikov measure only in the interior of the interval $I$. If $I$ is
an interval open on the left then the limit is in fact a Vishik-Fursikov
measure over the whole interval $I$. The problem is when $I$ is closed and
bounded on the left, since in this case we may lose the continuity at the
left end point of $I$ of the strengthened mean kinetic energy for the limit
measure (see condition $(iii)$ of the Definition \ref{vfmdef}). This problem is, in fact, inherited from an analogous problem for individual weak solutions, as described in Section \ref{sec1.5}, and which led us to introduce the spaces $\mU_I^\sharp(R)$ and $\mU_I^\sharp$.

In the next
result
we impose some conditions over a family of
$\alpha$-Vishik-Fursikov measures
in order to have that the limit is in fact a Vishik-Fursikov measure over
the
whole interval $I$. 

In fact we impose conditions only on the initial measures
$\Pi_{t_0}\rho_\alpha$, where $t_0$ is the left end point of the interval $I$.
These initial measures shall converge in a slightly stronger sense to a measure
$\mu_0$ in $H_w$. Since the $\sigma$-algebra of Borel sets with respect to the
weak topology of $H$ coincides with that with respect to the strong topology we
may consider $\mu_0$ as a Borel probability either on $H$ or on $H_w$. 

\medskip
\begin{theorem}\label{mtheo6}Let $I$ be an interval in $\mathbb R$ 
bounded and closed on the left, with left end point $t_0$, and let
$\{\rho_{\alpha}\}_{\alpha>0}$ be a family of $\alpha$-Vishik-Fursikov measures
over $I$. In addition, suppose that there
exists a Borel probability measure $\mu_0$ on $H$ such that
\begin{itemize}
\item [(i)] $\int_{H}|\bu|^2d\mu_0(\bu)<\infty$;
\item [(ii)] $\Pi_{t_0}\rho_{\alpha}\stackrel{*}{\rightharpoonup} \mu_0$ in
$\mathcal P(H_w)$;
\item [(iii)] and for all $\psi\in \Psi$
\[\lim_{\alpha\rightarrow
0}\int_{\mathcal Y_I}\psi(|\bu(t_0)|^2)d\rho_{\alpha}(\bu)=\int_H\psi(|\bu|
^2)d\mu_0(\bu).\] 
\end{itemize}
Then, there exists a sequence
$\{\rho_{\alpha_{n}}\}_n\subset \{\rho_{\alpha}\}_{\alpha>0}$, with
$\alpha_n\rightarrow 0$, that converges to a  Borel
probability measure $\rho$ in $\mathcal Y_I$ such that $\rho$ is a
Vishik-Fursikov measure over $I$ and $\Pi_{t_0}\rho=\mu_0$. 
\end{theorem}

\begin{proof}
Observe that by taking $\psi \equiv 1$ we can see that the family
$\{\rho_{\alpha}\}_{\alpha>0}$ fulfills the hypothesis of Corollary \ref{cor6}.
Thus, there exist a sequence
$\{\rho_{{\alpha_n}}\}_n\subset \{\rho_{\alpha}\}_{\alpha>0}$, with
$\alpha_n\rightarrow 0$, and Borel probability measure $\rho$ in $\mathcal
Y_I$ such that $\rho_{\alpha_{n}}\stackrel{*}{\rightharpoonup} \rho$
in $\mathcal P(\mathcal Y_I)$, as $n\rightarrow \infty$, $\rho$ is carried by
$\mathcal U_I^\sharp$, and $\Pi_{\mathring
I}\rho$ is a Vishik-Fursikov measure over
$\mathring I$. Then, in order to obtain
that $\rho$ is a Vishik-Fursikov measure over $I$, it remains to prove that
for all $\psi\in \Psi$ it holds true that
\begin{equation}\label{relation1}\lim_{t\rightarrow t_0^+} \int_{\mathcal
U^\sharp_I}\psi(|\bu(t)|^2)d\rho(\bu)=\int_{\mathcal
U^\sharp_I}\psi(|\bu(t_0)|^2)d\rho(\bu).
\end{equation}

Observe that for any $\phi \in \mathcal
C_b(H_w)$ the function defined by $\varphi= \phi\circ \Pi_{t_0}|_{\mathcal Y_I}$
belongs to $\varphi \in \mathcal
C_b(\mathcal Y_I)$. Then, it is straightforward from the convergence of
$\rho_{\alpha_n}$, from hypothesis $(ii)$ and from the Change of Variables
Theorem that 
\[\int_{H}\phi(\bu)d\mu_0(\bu)=\int_{\mathcal
Y_I}\phi(\bu(t_0))d\rho(\bu), \mbox{ for all } \phi \in \mathcal
C_b(H_w).\]
Therefore, since $H_w$ is a completely regular Hausdorff space and $\mu_0, 
\Pi_{t_0}\rho \in \mathcal P(H_w;t)$  we obtain
that $\Pi_{t_0}\rho=\mu_0$ (see \eqref{uniq-measu}).

Now, take $\psi\in\Psi$ and observe that, since any $\bu\in \mathcal
U^\sharp_I$ is weakly continuous at $t_0$ and $\psi$ is nondecreasing and
continuous, then $\psi(|\bu(t_0)|^2)\leq \liminf_{t\rightarrow
t_0}\psi(|\bu(t)|^2)$. Therefore, it follows from Fatou's Lemma that 
\begin{equation}\label{liminf}
\int_{\mathcal
U^\sharp_I}\psi(|\bu(t_0)|^2)d\rho(\bu)\leq \liminf_{t\rightarrow
t_0}\int_{\mathcal
U^\sharp_I}\psi(|\bu(t)|^2)d\rho(\bu).
\end{equation}
Let $\phi_M$ be a function in $\mathcal C^1([0,\infty))$ defined as
$\phi_M(r)=1$, for all $0\leq r\leq M$, $\phi_M(r)=0$,
for all $r\geq 2M$, $\phi_{M}\leq
\phi_{M+1}$, and $0\leq \phi_{M} \leq 1$, for all $M\in \mathbb N$. It is clear
that
$\varphi_m(\bu):=\psi(|P_m\bu(t)|^2)\phi_M(|P_m\bu(t)|^2)$
belongs to $\mathcal
C_b(\mathcal Y_I)$, for any $t\in I$, so that
\begin{equation*}
 \int_{\mathcal
Y_I}\psi(|P_m\bu(t)|^2)\phi_M(|P_m\bu(t)|^2)d\rho(\bu)=\lim_{
n\rightarrow \infty}\int_{\mathcal
Y_I}\psi(|P_m\bu(t)|^2)\phi_M(|P_m\bu(t)|^2)d\rho_{\alpha_n}(\bu).
\end{equation*}

Since $\phi_M\leq 1$ and $\psi(|P_m\bu(t)|^2)\leq
\psi(|\bu(t)|^2)$ we find that the right hand side of the last identity is
bounded above by $\limsup_{n\rightarrow
\infty}\int_{\mathcal Y_I}\psi(|\bu(t)|^2)d\rho_{\alpha_n}(\bu)$. On the other
hand,
using the Monotone Convergence Theorem twice we obtain that
\[\lim_{m\rightarrow \infty}\lim_{M\rightarrow \infty}\int_{\mathcal
Y_I}\psi(|P_m\bu(t)|^2)\phi_M(|P_m\bu(t)|^2)d\rho(\bu)=\int_{\mathcal
Y_I}\psi(|\bu(t)|^2)d\rho(\bu).\]
Therefore, for any $t\in I$, 
\begin{eqnarray}\label{ineq1}
\int_{\mathcal Y_I}\psi(|\bu(t)|^2)d\rho(\bu)\leq \limsup_{n\rightarrow
\infty}\int_{\mathcal Y_I}\psi(|\bu(t)|^2)d\rho_{\alpha_n}(\bu).
\end{eqnarray}
Now, using inequality \eqref{ineq1}, estimate \eqref{streng-ener} and the facts
that $\rho_\alpha$ is carried by $\mathcal U_I^\alpha$ and 
$\psi$ is nondecreasing, we obtain
\begin{equation*}
\begin{split}
&\limsup_{t\rightarrow t_0}\int_{\mathcal Y_I}\psi(|\bu(t)|^2)d\rho(\bu)\leq
\limsup_{t\rightarrow
t_0}\limsup_{n\rightarrow \infty}\int_{\mathcal
Y_I}\psi(|\bu(t)|^2)d\rho_{\alpha_n}(\bu)\\&
\leq \limsup_{t\rightarrow
t_0}\limsup_{n\rightarrow \infty}\left(\int_{\mathcal
Y_I}\psi(|\bu(t_0)|^2)d\rho_{\alpha_n}(\bu)+\frac{1}{
\lambda_1\nu}\|f\|_{L^\infty(I,H)}\sup_{r\geq
0}\psi'(r)(t-t_0)\right)\\ &
\leq \limsup_{n\rightarrow \infty}\int_{\mathcal
Y_I}\psi(|\bu(t_0)|^2)d\rho_{\alpha_n}(\bu)\leq 
\int_H\psi(|\bu|^2)d\mu_0(\bu),
\end{split}
\end{equation*}
where the last inequality follows from hypothesis $(iii)$.

To conclude, we use that $\Pi_{t_0}\rho=\mu_0$ in the last inequality together
with \eqref{liminf} to obtain that (\ref{relation1}) holds true.
\end{proof}

Now, we write the previous result in terms of statistical solutions.
\medskip
\begin{theorem}\label{mtheo6nv}Let $I$ be an interval in $\mathbb R$ 
bounded and closed on the left, with left end point $t_0$, and let
$\{\{\mu_t^{\alpha}\}_{t\in I}\}_{\alpha>0}$ be a family of 
$\alpha$-Vishik-Fursikov statistical solutions over $I$. In
addition, suppose that there
exists a Borel probability measure $\mu_0$ on $H$ such that
\begin{itemize}
\item [(i)] $\int_{H}|\bu|^2d\mu_0(\bu)<\infty$;
\item [(ii)] $\mu_{t_0}^{\alpha}\stackrel{*}{\rightharpoonup} \mu_0$ in
$\mathcal P(H_w)$;
\item [(iii)] and for all $\psi\in \Psi$
\[\lim_{\alpha\rightarrow
0}\int_{H}\psi(|\bu|^2)d\mu_{t_0}^{\alpha}(\bu)=\int_H\psi(|\bu|
^2)d\mu_0(\bu).\] 
\end{itemize}
Then, there exists a sequence $\{\{\mu_t^{\alpha_n}\}_{t\in I}\}_n \subset
\{\{\mu_t^{\alpha}\}_{t\in I}\}_{\alpha>0}$, with
$\alpha_n\rightarrow 0$, that converges  to a Vishik-Fursikov
statistical solution
$\{\mu_t\}_{t\in I}$ such that $\mu_{t_0}=\mu_0$.
\end{theorem}
\begin{proof}
It follows from the definition of $\alpha$-Vishik-Fursikov
statistical solution the existence of an
$\alpha$-Vishik-Fursikov measure
$\rho_{\alpha}$ over $I$ such that $\mu_t^{\alpha}=\Pi_t\rho_{\alpha}$,
for all $t\in I$. It is clear that the family
$\{\rho_{\alpha}\}_{\alpha>0}$ fulfills
the hypothesis of Theorem \ref{mtheo6} so that there exists a sequence
$\{\rho_{\alpha_{n}}\}_n\subset \{\rho_{\alpha}\}_{\alpha>0}$ that converges
to a Vishik-Fursikov measure $\rho$
over $I$. Therefore, the corresponding
sequence $\{\{\mu_t^{\alpha_{n}}\}_{t\in I}\}_n$ converges to
the Vishik-Fursikov statistical solution $\{\mu_t\}_{t\in I}$, where
$\mu_t=\Pi_t\rho$, for all $t\in I$.
\end{proof}

A case of particular interest for approximation purposes is when the measures
$\rho_\alpha$ have the same initial projection $\Pi_{t_0}\rho_\alpha=\mu_0$,
for all $\alpha$. This leads us to the following two corollaries, in terms of
Vishik-Fursikov measures and Vishik-Fursikov statistical solutions,
respectively.

\medskip
\begin{corollary}\label{maincortimedepss} Let $I$ be an interval in $\mathbb R$
closed and bounded on the left, with left end point $t_0$, and $\mu_0$ a Borel
probability on $H$ such that $\int_H|\bu|^2d\mu_0(\bu)<\infty$. Let
$\{\rho_{\alpha}\}_{\alpha>0}$ be a family of $\alpha$-Vishik-Fursikov measures
over $I$ such that $\Pi_{t_0}\rho_{\alpha}=\mu_0$, for all $\alpha>0$. Then,
there exists a sequence $\{\rho_{\alpha_{n}}\}_n\subset
\{\rho_{\alpha}\}_{\alpha>0}$, with $\alpha_n\rightarrow 0$, that converges to
a Borel probability
measure $\rho$ in $\mathcal Y_I$ such that $\rho$ is a Vishik-Fursikov measure
over $I$ and $\Pi_{t_0}\rho=\mu_0$. 
\end{corollary}

\medskip
\begin{corollary}\label{maincortimedepssnv} Let $I$ be an interval in $\mathbb
R$ closed and bounded on the left, with left end point $t_0$, and $\mu_0$ a
Borel probability on $H$ such that $\int_H|\bu|^2d\mu_0(\bu)<\infty$. Let
$\{\{\mu_t^{\alpha}\}_{t\in I}\}_{\alpha>0}$ be a family of 
$\alpha$-Vishik-Fursikov statistical solutions over $I$ such that
$\mu_{t_0}^{\alpha}=\mu_0$, for all $\alpha>0$. Then, there exists a sequence 
$\{\{\Pi_t\rho_{\alpha_n}\}_{t\in I}\}_n\subset \{\{\mu_t^{\alpha}\}_{t\in
I}\}_{\alpha>0}$, with $\alpha_n\rightarrow 0$, that converges to a
Vishik-Fursikov statistical solution $\{\mu_t\}_{t\in I}$ such that
$\mu_{t_0}=\mu_0$. 
\end{corollary}

\subsection{Stationary statistical solution}

We say that a
family $\{\mu_{\alpha}\}_{\alpha>0}$ of
stationary $\alpha$-Vishik-Fursikov statistical
solutions converges to a stationary Vishik-Fursikov statistical
solution if there exists a Borel probability measure $\mu$ on $H$ such that
$\mu$ is a stationary Vishik-Fursikov statistical solution and 
\[\lim_{\alpha\rightarrow
0}\int_H\phi(\bu)\rd\mu_{\alpha}(\bu)=\int_{H}\phi(\bu)\rd \mu(\bu),\]
for all $\phi\in \mathcal C_b(H_w)$.

\medskip
\begin{theorem}\label{mainthmstatss}
Let $\{\mu_{\alpha}\}_{\alpha>0}$ be a family such that, for each $\alpha>0$,
$\mu_{\alpha}$ is a stationary $\alpha$-Vishik-Fursikov statistical
solution. Suppose that 
$\int_H|\bu|^2d\mu_{\alpha}(\bu)\leq C$, for all
$\alpha>0$, for some $C\geq 0$. Then, there exists a sequence 
$\{\mu_{\alpha_n}\}_n\subset \{\mu_{\alpha}\}_{\alpha>0}$, with
$\alpha_n\rightarrow 0$, that converges  to a Borel
probability measure $\mu$ on $H$.
Moreover, $\mu$ is a stationary Vishik-Fursikov statistical solution.
\end{theorem}

\begin{proof}
Since $\mu_{\alpha}$ is a stationary $\alpha$-Vishik-Fursikov statistical
solution, for each $\alpha>0$, there exists an invariant Vishik-Fursikov
measure $\rho_{\alpha}$ over an interval $I$ unbounded on the right, such
that $\mu_{\alpha}=\Pi_t\rho_{\alpha}$, at any time $t\in I$. Hence, using
the Change of Variables Theorem it follows that
\[\int_{\mathcal
Y_I}|\bu(t)|^2\rd\rho_{\alpha}(\bu)=\int_{H}|\bu|^2\rd\mu_{\alpha}(\bu), \;
\forall t\in I, \; \forall \alpha>0.\]
From the hypothesis $\int_H|\bu|^2\rd\mu_{\alpha}(\bu)\leq C$, for
all $\alpha>0$, and the last identity, we obtain that
\[\sup_{t\in I}\int_{\mathcal
Y_I}|\bu(t)|^2\rd\rho_{\alpha}(\bu)\leq C,\; \forall \alpha>0.\]
Thus, we can apply Proposition
\ref{prop5} to obtain a sequence,
$\{\rho_{\alpha_n}\}_n\subset \{\rho_{\alpha}\}_{\alpha>0}$, with
$\alpha_n\rightarrow 0$, and a Vishik-Fursikov measure $\rho$ over $I$ such
that
$\rho_{\alpha_n} \stackrel{*}{\rightharpoonup} \rho$ in $\mathcal P(\mathcal
Y_I)$. First, let us check that $\rho$ is an invariant Vishik-Fursikov
measure, which by \eqref{uniq-measu} is equivalent to show that 
\[\int_{\mathcal Y_I}\varphi(\sigma_\tau\bu)\rd
\rho(\bu)=\int_{\mathcal Y_I}\varphi(\bu)\rd \rho(\bu), \mbox{ for all
}\varphi\in \mathcal C_b(\mathcal Y_I).\]
We already have that for all $\varphi\in \mathcal C_b(\mathcal Y_I)$,
\[\int_{\mathcal Y_I}\varphi(\sigma_\tau\bu)\rd
\rho_{\alpha_n}(\bu)=\int_{\mathcal Y_I}\varphi(\bu)\rd \rho_{\alpha_n}(\bu).\]
Note also that $\varphi\circ \sigma_\tau$ belongs to  $\mathcal
C_b(\mathcal Y_I)$ for all  $\varphi\in \mathcal C_b(\mathcal Y_I)$, so that,
since $\rho_{\alpha_n} \stackrel{*}{\rightharpoonup} \rho$ in $\mathcal
P(\mathcal Y_I)$, we obtain that
\begin{equation*}
\int_{\mathcal Y_I}\varphi(\sigma_\tau\bu)\rd
\rho(\bu)=\lim_{n\rightarrow \infty}\int_{\mathcal
Y_I}\varphi(\sigma_\tau\bu)\rd
\rho_{\alpha_n}(\bu)=\lim_{n\rightarrow \infty}\int_{\mathcal
Y_I}\varphi(\bu)\rd \rho_{\alpha_n}(\bu)=\int_{\mathcal
Y_I}\varphi(\bu)\rd \rho(\bu).
\end{equation*}

Define $\mu:=\Pi_t\rho$, for an arbitrary time $t\in I$. Then, $\mu$ is a
stationary Vishik-Fursikov
statistical solution by definition. Moreover, $\mu_{\alpha_n}
\stackrel{*}{\rightharpoonup} \mu$ in $\mathcal P(H_w)$. Indeed, let $\phi\in
\mathcal C_b(H_w)$, then $\varphi:=\phi\circ \Pi_t$ belongs to $\mathcal
C_b(\mathcal Y_I)$ and  
\[\lim_{n\rightarrow
\infty}\int_H\phi(\bu)\rd\mu_{\alpha_n}(\bu)=\lim_{n\rightarrow
\infty}\int_{\mathcal Y_I}\varphi(\bu)\rd
\rho_{\alpha_n}(\bu)=\int_{\mathcal Y_I}\varphi(\bu)\rd
\rho(\bu)=\int_{H}\phi(\bu)\rd \mu(\bu).\]
This concludes the proof.
\end{proof}

\section{Conclusions}
We have proved that, under natural conditions, families of statistical solutions
of the 3D Navier-Stokes-$\alpha$ model, depending on the parameter $\alpha$,
possess, as $\alpha$ goes to zero,  subsequences that converge to a statistical
solution of the 3D Navier-Stokes equations. The main condition for the
convergence is that the mean kinetic energy of the family of statistical
solutions be uniformly bounded. This yields suitable tightness and compactness
properties needed for the existence of a convergent subsequence. 

The statistical solutions contain the statistical information of a given flow,
including turbulent flows, which are the case of most interest. It is therefore
natural to study conditions that guarantee that the statistical information
obtained from an approximate problem converges to the original problem in a
suitable sense. We succeeded in proving this convergence under simple and
natural conditions. This result implies that the statistical information
obtained from the 3D Navier-Stokes-$\alpha$ model are good approximations of the
statistical information of flows modelled by the 3D Navier-Stokes equations.

Moreover, the techniques developed in this paper seem to allow for an extension
of the result to a wide range of models and approximations. This is currently a
work in progress and will be presented elsewhere. 

\medskip
{\bf Acknowledgements.} The authors would like to thank Cec\'{\i}lia
Mondaini for
enriching the discussions related to this work and F\'abio Ramos for drawing our
attention to the work of Topsoe. This work was partly supported by CNPq,
Bras\'{\i}lia, Brazil, under the grants  \#151082/2010-3, \#307077/2009-8,
\#477379/2010-9.
\medskip


\end{document}